%% file: complex.tex
\documentclass[preprint,12pt]{elsarticle}

\usepackage{graphicx}

\usepackage{amssymb}
\usepackage{amsthm}
\usepackage{amsmath}
\usepackage{amsfonts}

\biboptions{comma,square}

\journal{Advances in Mathematics}

\usepackage[breaklinks,pdfstartview=FitH,colorlinks,linktocpage,bookmarks=true]{hyperref}
\usepackage{geometry}
\usepackage{pgf,tikz}\usetikzlibrary{arrows}

\newtheorem{theorem}{Theorem}[section]
\newtheorem{uniqueness-theorem}[theorem]{Existence and Uniqueness Theorem}
\newtheorem{convergence-theorem}[theorem]{Convergence Theorem}
\newtheorem{variational-principle}[theorem]{Variational Principle}
\newtheorem{convexity-principle}[theorem]{Convexity Principle}
\newtheorem{energy-conservation-principle}[theorem]{Energy Conservation Principle}
\newtheorem{maximum-principle}[theorem]{Maximum Principle}
\newtheorem{energy-convergence-lemma}[theorem]{Energy Convergence Lemma}
\newtheorem{equicontinuity-lemma}[theorem]{Equicontinuity
Lemma}
\newtheorem{continuity-lemma}[theorem]{Equicontinuity 
Lemma}
\newtheorem{friedrichs-lemma}[theorem]{Friedrichs Inequality Lemma}
\newtheorem{limit-harmonicity-lemma}[theorem]{Lemma}
\newtheorem{laplacian-approximation-lemma}[theorem]{Laplacian Approximation Lemma}
\newtheorem{Green-identity}[theorem]{Green Identity}
\newtheorem{path-energy-estimate}[theorem]{Path Energy Lemma}
\newtheorem{projection-lemma}[theorem]{Projection Lemma}
\newtheorem{lemma-on-vertices-in-rectangle}[theorem]
{Rectangle Capacity Lemma}
\newtheorem{diameter-lemma}[theorem]{Diameter Lemma}
\newtheorem{gradient-approximation-lemma}[theorem]{Gradient Approximation Lemma}

\newtheorem{lemma}[theorem]{Lemma}

\newtheorem{corollary}[theorem]{Corollary}

\theoremstyle{definition}

\newtheorem{example}[theorem]{Example}
\newtheorem{problem}[theorem]{Problem}

\theoremstyle{remark}
\newtheorem{remark}[theorem]{Remark}

\DeclareMathOperator{\Real}{Re}
\DeclareMathOperator{\Imaginary}{Im}

\DeclareMathOperator{\Closure}{Cl}

\begin{document}

\begin{frontmatter}



\title{The boundary value problem for discrete analytic functions}


\author[add2,add3]
{M. Skopenkov\fnref{fn}}
\ead{skopenkov@rambler.ru} 
\ead[URL]{http://skopenkov.ru} 
\address[add2]{Institute for Information Transmission Problems of the Russian Academy of Sciences,\\
Bolshoy Karetny per.~19, bld.~1, Moscow, 127994, Russian Federation}
\address[add3]{King Abdullah University of Science and Technology, Saudi Arabia}
\fntext[fn]{The author was partially supported by ``Dynasty'' foundation, by the Simons--IUM fellowship, and by the President of the Russian Federation grant MK-3965.2012.1.}


\begin{abstract}
This paper is on further development of discrete complex analysis introduced by R.~Isaacs, J.~Ferrand, R.~Duffin, and C.~Mercat.
We consider a graph lying in the complex plane and having quadrilateral faces.  A function on the vertices is called discrete analytic, if for each face the difference quotients along the two diagonals are equal.

We prove that the Dirichlet boundary value problem for the real part of a discrete analytic function has a unique solution. In the case when each face has orthogonal diagonals we prove that this solution uniformly converges to a harmonic function in the scaling limit. 
This solves a problem of S.~Smirnov from 2010. 
This was proved earlier by R.~Courant--K.~Friedrichs--H.~Lewy and L.~Lusternik for square lattices, by D.~Chelkak--S.~Smirnov and implicitly by P.G.~Ciarlet--P.-A.~Raviart for rhombic lattices.

In particular, our result implies uniform convergence of the finite element method on Delaunay triangulations. This solves a problem of A.~Bobenko from 2011.
The methodology is based on energy estimates inspired by alternating-current network theory.
\end{abstract}


\end{frontmatter}

\section{Introduction}\label{sec-intro}

Various constructions of complex analysis
on planar graphs were introduced by Isaacs, Ferrand, Duffin, Mercat \cite{Isaacs-41, Ferrand-44, Duffin-68, Mercat-01, Mercat-08}, Dynnikov--Novikov \cite{Dynnikov-Novikov-03}, Bobenko--Mercat--Suris~\cite{Bobenko-etal-05}, and
Bobenko--Pinkall--Springborn~\cite{Bobenko-etal-10}.
Recently this subject is developed extensively due to applications to statistical physics \cite{Smirnov-10}, numerical analysis \cite{Hilderbrandt-etal-06, Bobenko-etal-09}, computer graphics \cite{Alexa-Wardetzky-11, Wardetzky-etal-07}, and combinatorial geometry \cite{Prasolov-Skopenkov-11}; see \cite{Lovasz-04, Smirnov-10} for recent surveys.

This paper concerns \emph{linear} complex analysis on quadrilateral lattices ~\cite{Bobenko-etal-05}. A \emph{quadrilateral lattice} 
is a finite graph $Q\subset\mathbb{C}$ with rectilinear edges such that each bounded face is a quadrilateral (not necessarily convex). Depending on the shape of faces, one speaks about \emph{square}, \emph{rhombic}, 
or \emph{orthogonal} lattices (the latter is quadrilateral lattices such that the diagonals of each face are orthogonal); see Figure~\ref{fig-meshes}. Different types of lattices are required for different applications; see Section~\ref{sec-open-problems}.




\vspace{-0.1cm}
\begin{figure}[htb]
\begin{center}
\begin{tabular}{cccc}
\input{dirichlet-figure1a.tex} &
\input{dirichlet-figure1b.tex} &
\input{dirichlet-figure1c.tex} &
\input{complex-figure3d.tex}
\end{tabular}
\end{center}
\vspace{-0.4cm}
\caption{Examples of lattices $Q$ (from the left to the right): square; rhombic; orthogonal;  quadrilateral.}
\label{fig-meshes}
\end{figure}
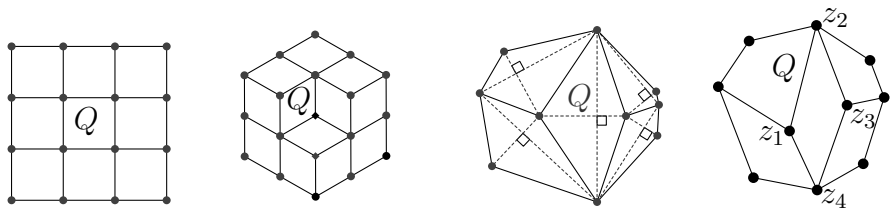

A complex-valued function $f$ on the vertices of $Q$ is called \emph{discrete analytic} \cite{Mercat-08}, if the difference quotients along the two diagonals of each face are equal, i.~e.,
\begin{equation}\label{eq-def-analytic}
\frac{f(z_1)-f(z_3)}{z_1-z_3}=
\frac{f(z_2)-f(z_4)}{z_2-z_4}
\end{equation}
for each quadrilateral face $z_1z_2z_3z_4$; 
see Figure~\ref{fig-meshes} to the right.
The motivation for this definition is that both sides of equation~\eqref{eq-def-analytic} approximate the derivative of an analytic function $f$ inside this face.
The real part of a discrete analytic function is called a \emph{discrete harmonic function}. 

Discrete complex analysis is analogous to the classical complex analysis in many aspects \cite{Lovasz-04}. One of the most natural and at the same time challenging problems is to prove convergence of discrete theory to the continuous one when the lattice becomes finer and finer \cite{Smirnov-10}. A natural formalization of such \emph{convergence} is uniform convergence of the solution of the Dirichlet boundary value problem for
a discrete harmonic function to a harmonic function in the scaling limit.

\subsection{Previous work}
Convergence in this sense was proved by R.~Courant--K.~Friedrichs--H.~Lewy \cite[\S4]{Courant-friedrichs-Lewy-28} and L.~Lusternik \cite[\S4--5]{Lusternik-26}
for square lattices,
by D.~Chelkak--S.~Smirnov \cite[Proposition~3.3]{Chelkak-Smirnov-08} and
(implicitly and in less general setup) by P.G.~Ciarlet--P.-A.~Raviart
\cite[Theorem~2]{Ciarlet-Raviart-73} for rhombic lattices.
In fact convergence for rhombic lattices is equivalent to convergence of the classical finite element method  \cite{Duffin-68}. The latter subject is well-developed; see a survey \cite{Brandts-etal-11} and a textbook \cite{Braess-07}. Nonrhombic lattices cannot be accessed by known methods. Weaker convergence results not involving boundary value problems were obtained in \cite[Theorem~3]{Mercat-01}, \cite[Theorem~2]{Hilderbrandt-etal-06}.


\subsection{Contributions}


We prove that the Dirichlet boundary value problem for 
a discrete harmonic function
on a quadrilateral lattice has a unique solution. Our main result is that in the case of orthogonal lattices this solution uniformly converges to a harmonic function in the scaling limit; see 
Convergence Theorem~\ref{th-uniform-approx} below.
This solves a problem of S.~Smirnov~\cite[Question~1]{Smirnov-10}.

In particular, our main result implies uniform convergence of the finite element method on Delaunay triangulations; see Corollary~\ref{cor-cotangent} below. This solves a problem of A.~Bobenko (private communication; see also \cite[Table in Section~2]{Wardetzky-etal-07}). Our main result also implies uniform convergence of the discrete harmonic measure; see Corollary~\ref{cor-measure} below.

\subsection{Statements} \label{ssec-state}

Let us give precise statements of main results.

The \emph{boundary} $\partial Q$ of the lattice $Q$ is the boundary of its outer face. Hereafter assume for simplicity that $\partial Q$ is a closed  curve without self-intersections. 
Assume that the intersection of any two faces of $Q$ is either empty, or a single vertex, or a single edge.
Denote by $Q^0$ the set of vertices of~$Q$. 

Let $g\colon\mathbb{C}\to\mathbb{R}$ be a smooth function. The \emph{Dirichlet (boundary value) problem on $Q$} is to find a discrete harmonic function $u_{Q,g}\colon Q^0\to\mathbb{R}$ such that $u_{Q,g}(z)=g(z)$ for each vertex $z\in\partial Q$. The function $u_{Q,g}\colon Q^0\to\mathbb{R}$ is called a \emph{solution} of the Dirichlet problem.

\begin{uniqueness-theorem}
\label{cl-dirichlet-problem}
The Dirichlet boundary value problem on any finite quadrilateral lattice has a unique solution.
\end{uniqueness-theorem}

This theorem is nontrivial because discrete harmonic functions do not satisfy the maximum principle in general; see Example~\ref{ex-nomax} below.


Let $\Omega\subset\mathbb{C}$ be a bounded simply-connected domain (no smoothness assumptions are imposed on the boundary  $\partial \Omega$). The \emph{Dirichlet (boundary value) problem on $\Omega$}
is to find a continuous function $u_{\Omega,g}\colon \Closure\Omega\to\mathbb{R}$ harmonic in $\Omega$
such that $u_{\Omega,g}(z)=g(z)$ for each point $z\in\partial \Omega$.
The function $u_{\Omega,g}\colon \mathrm{Cl}\Omega\to\mathbb{R}$ is called a \emph{solution} of the Dirichlet problem. It is well-known that the solution always exists and is unique; e.g., see \cite[Section~3.3]{Chelkak-Smirnov-08}.


A sequence of lattices $\{Q_n\}$ \emph{approximates the domain} $\Omega$, if for $n\to\infty$:
\begin{itemize}
\item the maximal distance from a point of $\partial Q_n$ to the set $\partial\Omega$ tends to zero;
\item the maximal distance from a point of $\partial\Omega$ to the set $\partial Q_n$ tends to zero;
\item the maximal edge length of $Q_n$ tends to zero.
\end{itemize}

A sequence of lattices $\{Q_n\}$ is \emph{nondegenerate uniform}, if there is a constant $\mathrm{Const}$ (not depending on $n$) such that for each member of the sequence:
\begin{itemize}
\item[(D)] the ratio of the diagonals of each face is less than $\mathrm{Const}$ and the angle between them is greater than $1/\mathrm{Const}$;
\item[(U)] the number of vertices in an arbitrary disk of radius equal to the maximal edge length is less than $\mathrm{Const}$.
\end{itemize}



%

A sequence of functions $u_{n}\colon Q^0_n\to\mathbb{R}$ \emph{uniformly converges} to a function $u\colon\mathrm{Cl}\Omega\to\mathbb{R}$, if for  $n\to\infty$ we have $\max_{z\in Q_n^0\cap\mathrm{Cl}\Omega}|u_n(z)-u(z)|\to 0$.

\begin{convergence-theorem}
\label{th-uniform-approx}
Let $\Omega\subset\mathbb{C}$ be a bounded simply-connected domain. Let $g\colon \mathbb{C}\to\mathbb{R}$ be a smooth function. Let $\{Q_n\}$ be a nondegenerate uniform sequence of finite orthogonal lattices  approximating the domain $\Omega$. Then the solution $u_{Q_n,g}\colon Q_n^0\to\mathbb{R}$ of the Dirichlet problem on $Q_n$ uniformly converges to the solution $u_{\Omega,g}\colon \mathrm{Cl}\Omega\to\mathbb{R}$ of the Dirichlet problem on $\Omega$.
\end{convergence-theorem}

\subsection{Organization of the paper}
In Section~\ref{sec-main-ideas} we introduce main ideas of the proofs and state key lemmas.
In Sections~\ref{sec-variational-principle} and~\ref{sec-scaling-limit} we prove Existence and Uniqueness Theorem~\ref{cl-dirichlet-problem} and
Convergence Theorem~\ref{th-uniform-approx}, respectively. In Section~\ref{sec-open-problems} we give applications of our results to numerical analysis, network theory, probability theory, and state some open problems.

\section{Main ideas}\label{sec-main-ideas}


\subsection{Energy minimization}\label{ssec-energy}

Our approach is based on energy estimates inspired by alternating-current network theory.
%

Recall that the \emph{(Dirichlet) energy} of a continuous piecewise-smooth function $u\colon \Omega\to\mathbb{R}$ is
$$
E_\Omega(u):=\int_\Omega |\nabla u|^2\, dxdy.
$$
This is a convex functional on the space of continuous piecewise-smooth functions with fixed boundary values, and harmonic functions are characterized as minimizers of this functional.

Let us define a discrete counterpart of the energy,
which is the main concept of the paper.
The \emph{gradient} of a function $u\colon Q^0\to \mathbb{R}$ at a face $z_1z_2z_3z_4$ of the  quadrilateral lattice $Q$ is the unique vector $\nabla_{\!Q} u(z_1z_2z_3z_4)\in\mathbb{R}^2$ such that
\begin{align*}
\nabla_{\!Q} u(z_1z_2z_3z_4) \cdot \overrightarrow{z_1z_3}&=u(z_3)-u(z_1), \\
\nabla_{\!Q} u(z_1z_2z_3z_4) \cdot \overrightarrow{z_2z_4}&=u(z_4)-u(z_2).
\end{align*}
The \emph{energy} of the function $u\colon Q^0\to \mathbb{R}$
is the number
\begin{equation}\label{eq-def-energy}
E_Q(u):=\sum_{z_1z_2z_3z_4\subset Q}|\nabla_{\!Q} u(z_1z_2z_3z_4)|^2\cdot \mathrm{Area}(z_1z_2z_3z_4),
\end{equation}
where the sum is over all the faces $z_1z_2z_3z_4$ of the lattice $Q$.
In the particular case of an orthogonal lattice the energy takes the usual form
\begin{equation*}
E_Q(u)=\sum_{z_1z_2z_3z_4\subset Q}
\left(c(z_1z_3)\left(u(z_3)-u(z_1)\right)^2+
c(z_1z_3)^{-1}\left(u(z_4)-u(z_2)\right)^2\right)/2.
\end{equation*}
Hereafter we list the vertices of each face $z_1z_2z_3z_4$ clockwise and denote $c(z_1z_3):=i\frac{z_2-z_4}{z_1-z_3}$.

We give a physical motivation for this definition in Section~\ref{ssec-phys}. A similar but nonequivalent definition was given in \cite[Formula~(12)]{Mercat-08}.
Our energy has the same properties as its continuous counterpart.

\begin{convexity-principle} \label{cl-min-energy}
The energy $E_Q(u)$ is a strictly convex functional on the affine space $\mathbb{R}^{Q^0-\partial Q}$ of functions $u\colon Q^0\to \mathbb{R}$ having fixed values at the boundary~$\partial Q$.
\end{convexity-principle}

\begin{variational-principle} \label{cl-laplacian} \label{cl-upper-bound-energy-inequality} A function $u\colon Q^0\to\mathbb{R}$ has minimal energy $E_Q(u)$ among all the functions with the same boundary values if and only if it is discrete harmonic.
\end{variational-principle}

These principles are proved in Section~\ref{sec-variational-principle}. Existence and Uniqueness Theorem~\ref{cl-dirichlet-problem} is their direct consequence. After the ``right'' discrete energy has been guessed, 
the proof of these results is standard.




\subsection{Energy estimates}\label{ssec-plan}

Let us state more delicate energy estimates required for the proof of Convergence Theorem~\ref{th-uniform-approx}.

Joining the opposite vertices in each quadrilateral face of the lattice $Q$, we get two connected graphs $B$ and~$W$ \emph{associated} to the lattice; see Figure~\ref{fig-q-lattice}. (The opposite vertices of a face are joined by a straight line segment, if the segment lies inside the face, and by the $2$-segment broken line through the midpoint of the opposite diagonal, otherwise.)
The \emph{eccentricity} of a lattice $Q$ is the infimum of the numbers $\mathrm{Const}$ such that the lattice satisfies conditions~(D) and~(U) from Section~\ref{ssec-state}.
Throughout the paper we use the following notation:
\begin{itemize}
\item $g\colon\mathbb{C}\to\mathbb{R}$ is an arbitrary smooth function; 
$|D^k g(z)|:=\max_{0\le j\le k} \left|\frac{\partial^k g}{\partial^j x\,\partial^{k-j}y}(z)\right|$;
\item $B$ and $W$ are the two graphs associated to the lattice $Q$;
\item $e$ is the eccentricity of the lattice $Q$;
\item $h$ is twice the maximal edge length of the lattice $Q$.
\end{itemize}

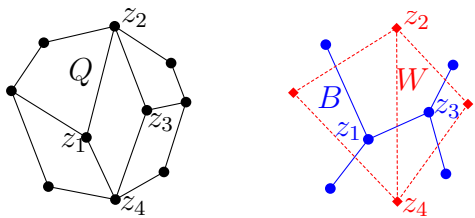
\begin{figure}[htbp]
\begin{center}
\input{complex-figure3v2.tex}\qquad\input{complex-figure1v2.tex}
\caption{The graphs $B$ and $W$ associated to a quadrilateral lattice $Q$.} 
\label{fig-q-lattice}
\end{center}
\end{figure}






Our first lemma gives the basis for our estimates.

\begin{energy-convergence-lemma}\label{l-energy-convergence}
Let $\Omega$ be a bounded simply-connected domain with smooth boundary. 
Let $\{Q_n\}$ be a nondegenerate uniform sequence of quadrilateral lattices approximating the domain $\Omega$. 
Then $E_{Q_n}(g\left|_{Q_n^0}\right.)\to E_{\Omega}(g)$ as
$n\to\infty$.
\end{energy-convergence-lemma}

Our main lemma estimates the difference between the values of a discrete harmonic function at two vertices of the graph~$B$ through the 
distance between them. Estimates of the form $|u(z)-u(w)|\le \mathrm{Const}\cdot |z-w|^p$ were known for square lattices, $p=1/2$ \cite[equation~(12) in \S4.2]{Courant-friedrichs-Lewy-28}, and for rhombic ones, $p=1$ \cite[Corollary~2.9, Proposition~2.7 and Appendix~A]{Chelkak-Smirnov-08}. 
However their proofs 
do not generalize to more general lattices (because in general there are no discrete exponentials \cite{Kenyon-02} and no higher derivatives of discrete analytic functions). Following the approach of \cite[\S5.4]{Lusternik-26} and~\cite[Corollary~3.4]{Saloff-Coste-97} we get an  estimate, which is harder to state but much easier to prove.



\begin{continuity-lemma}\label{l-equicontinuity}
Let $Q$ be an orthogonal lattice. 
Let $u\colon Q^0\to \mathbb{R}$ be a discrete harmonic function.
Let $z,w\in B^0$ be two vertices with $|z-w|\ge h$.
Let $R$ be a square of side length $r>3|z-w|$ with the center at the midpoint of the segment $zw$ and the sides parallel and orthogonal to $zw$.
Then there is a constant $\mathrm{Const}_{e}$  depending only on $e$ (but not on $Q$, $u$, $z$, $w$, $R$, $r$, $h$) 
such that 
\begin{equation}\label{eq-l-equicontinuity}
\left|u(z)-u(w)\right|\le
\mathrm{Const}_{e}\cdot E_Q(u)^{1/2}\cdot \ln^{-1/2}\frac{r}{3|z-w|}+\max_{z',w'\in R\cap \partial Q\cap B^0}\left|u(z')-u(w')\right|,
\end{equation}
where the maximum is over all pairs of boundary vertices $z',w'\in  \partial Q\cap B^0$ lying inside $R$. The maximum is set to be zero, if there are no such vertices. 
For $|z-w|< h< r/3$ the same inequality holds with $|z-w|$ replaced by $h$.
\end{continuity-lemma}



Our next lemma concerns approximation of the laplacian. In the linear space of functions $u\colon Q^0\to\mathbb{R}$, set the coordinates to be the values of the function. The energy $E_Q(u)$ is quadratic in these coordinates. Define the \emph{laplacian} $\Delta_Q u\colon Q^0\to\mathbb{R}$ of a function $u\colon Q^0\to\mathbb{R}$ by the formula
\begin{equation*} 
[\Delta_Q u](z):=-\frac{\partial E_Q(u)}{\partial u(z)} \qquad \text{for each $z\in Q^0$.} \end{equation*}
In the particular case of orthogonal lattices the laplacian takes the usual form
\begin{equation*}
[\Delta_Q u](z)=\sum_{z_1z_3\subset B\,:\,z_1=z} c(z_1z_3)\left(u(z_3)-u(z_1)\right).
\end{equation*}
In what follows we omit the arguments of $\Delta_Q u$ and $\nabla_{\!Q} u$, if no confusion arises. For rhombic lattices $\Delta_Q u$ approximates the laplacian $\Delta u$ in some sense \cite[Lemma~2.2]{Chelkak-Smirnov-08}. For nonrhombic lattices such approximation does not hold. 
Similarly to~\cite{Hilderbrandt-etal-11}, we use \emph{mean} approximation of the laplacian.

\begin{laplacian-approximation-lemma}\label{cl-approximation}
Let $Q$ be a 
quadrilateral lattice and $R$ be a square of side length $r>h$ inside $\partial Q$.
Then there is a constant $\mathrm{Const}_e$ depending only on $e$ (but not on $Q$, $g$, $R$, $r$, $h$) 
such that
\begin{equation*}
\left|
\sum_{z\in R\cap B^0} \left[\Delta_Q (g\left|_{Q^0}\right.)\right](z)-
\int_R \Delta g\, dxdy
\right|\le
\mathrm{Const}_e\left(hr\max_{z\in R}|D^2 g(z)|+r^3\max_{z\in R}|D^3 g(z)|\right).
\end{equation*}
\end{laplacian-approximation-lemma}

Lemmas~\ref{l-energy-convergence}--\ref{cl-approximation}
 are proved in Section~\ref{sec-scaling-limit} using suitable modifications of the approaches of \cite{Duffin-59}, \cite{Lusternik-26},
 \cite{Chelkak-Smirnov-08}, 
 respectively.
Overall scheme of the proof of Convergence Theorem~\ref{th-uniform-approx} is analogous to \cite{Courant-friedrichs-Lewy-28, Chelkak-Smirnov-08}. 

\begin{proof}[Sketch of the proof of Convergence Theorem~\ref{th-uniform-approx} modulo the above lemmas]
(See details in Section~\ref{ssec-convergence}.)
Restrict each function $u_{Q_n,g}$ to 
$B^0_n$. 
By Variational Principle~\ref{cl-laplacian} and Energy Convergence Lemma~\ref{l-energy-convergence} we have
$E_{Q_n}(u_{Q_n,g})\le E_{Q_n}(g\left|_{Q^0_n}\right.)\to E_\Omega(g)<\infty$ for $n\to\infty$. Thus 
$E_{Q_n}(u_{Q_n,g})$ is bounded. Then for $r:=\sqrt{|z-w|}$ 
the right-hand side of~\eqref{eq-l-equicontinuity} tends to $0$ as $|z-w|\to 0$. By Equicontinuity Lemma~\ref{l-equicontinuity}
%
the sequence 
$u_{Q_n,g}$ is equicontinuous. By the Arzel\`a-Ascoli theorem a subsequence $u_{Q_k,g}$ of the sequence $u_{Q_n,g}$ uniformly converges  to a continuous function $u\colon \mathrm{Cl}\,\Omega\to\mathbb{R}$. By continuity 
$u(z)=g(z)$ for each $z\in\partial \Omega$.
Using the Weyl lemma and Laplacian Approximation Lemma~\ref{cl-approximation} for $r:=\sqrt{h}$ we show that 
$u$ is harmonic in $\Omega$. 
Hence $u$ equals the unique solution $u_{\Omega,g}$ of the Dirichlet problem. 
Thus the initial sequence $u_{Q_n,g}$ (not just a subsequence)
 uniformly 
converges to $u_{\Omega,g}$.
\end{proof}

\section{Existence and Uniqueness}\label{sec-variational-principle}

In this section we prove 
Existence and Uniqueness Theorem~\ref{cl-dirichlet-problem}
and the results stated in Section~\ref{ssec-energy}.
We also prove two results (Maximum Principle~\ref{cl-maximum-principle} and Green's Identity~\ref{cl-green} below) 
required for the next section.

\subsection{Convexity Principle}

\begin{proof}[Proof of Convexity Principle~\ref{cl-min-energy}]
Consider the linear space $\mathbb{R}^{Q^0}$ of functions $u\colon Q^0\to\mathbb{R}$. 
Clearly, the gradient $\nabla_{\!Q} u$ linearly depends on $u$ and thus the energy $E_Q(u)$ is a quadratic form in $u$. So it suffices to prove the convexity of $E_Q(u)$ in the case when
the affine space $\mathbb{R}^{Q^0-\partial Q}\subset \mathbb{R}^{Q^0}$ passes through the origin, that is, all the fixed boundary values equal zero.

Clearly, $E_Q(u)\ge 0$ for each $u\in \mathbb{R}^{Q^0-\partial Q}$. It remains to prove that $E_Q(u)=0$ only if $u=0$. Assume that $E_Q(u)=0$. Then $\nabla_{\!Q} u=0$.
Thus for each face $z_1z_2z_3z_4$ we have $u(z_1)=u(z_3)$ and $u(z_2)=u(z_4)$. 
Any face can be joined with the boundary $\partial Q$ by a sequence of faces such that two neighboring ones share a common edge.
Thus for each $z\in Q^0$ there is $w\in\partial Q$ such that $u(z)=u(w)$. 
Thus $u=0$.
\end{proof}

\subsection{Variational Principle}\label{ssec-variational}






Denote by $*\colon\mathbb{R}^2\to\mathbb{R}^2$ the counterclockwise rotation through $\pi/2$ about the origin. 
Two functions $u,v\colon Q^0\to\mathbb{R}$ are \emph{conjugate}, if $\nabla_{\!Q} v=*\nabla_{\!Q} u$. 


\begin{lemma} \label{cl-analyticity} Two functions $u,v\colon Q^0\to\mathbb{R}$ are conjugate if and only if the function $u+iv\colon Q^0\to\mathbb{C}$ is discrete analytic.
\end{lemma}

\begin{proof}
Identify the gradient $\nabla_{\!Q} u\in\mathbb{R}^2$ with a complex number $\nabla_{\!Q} u\in\mathbb{C}$.
The function $u+iv\colon Q^0\to\mathbb{C}$ is discrete analytic if and only if
$$
\frac{u(z_1)+iv(z_1)-u(z_3)-iv(z_3)}{z_1-z_3}=
\frac{u(z_2)+iv(z_2)-u(z_4)-iv(z_4)}{z_2-z_4}.
$$
for each face $z_1z_2z_3z_4$ of the lattice $Q$.
Substitute the expression
\begin{equation*}\label{eq-nabla}
u(z_1)-u(z_3)=
\nabla_{\!Q} u\cdot\overline{(z_1-z_3)}/2
+\overline{\nabla_{\!Q} u}\cdot(z_1-z_3)/2
\end{equation*}
and analogous ones for $v(z_1)-v(z_3)$, $u(z_2)-u(z_4)$, $v(z_2)-v(z_4)$ into the above equation. We get 
$$
\left(\nabla_{\!Q} u+i\nabla_{\!Q} v\right)\cdot\left( \frac{\overline{z_1-z_3}}{z_1-z_3}-\frac{\overline{z_2-z_4}}{z_2-z_4}\right)=0.
$$
Since the second factor on the left-hand side is nonzero, the equation is equivalent to
$ \nabla_{\!Q} v=i\nabla_{\!Q} u.$
\end{proof}

\begin{lemma}\label{cl-exist-conj2} A function $u\colon Q^0\to\mathbb{R}$ has a conjugate if and only if
for each $z\in Q^0-\partial Q$ we have
$$\sum_{z_1z_2z_3z_4\,:\,z_1=z} *\nabla_{\!Q} u(z_1z_2z_3z_4) \cdot\overrightarrow{z_4z_2}=0,$$
where the sum is over all the faces $z_1z_2z_3z_4$ of the lattice $Q$ such that $z_1=z$.
\end{lemma}

\begin{proof}
Let us prove the ``only if'' part.
Assume that $v\colon Q^0\to\mathbb{R}$ is conjugate to $u$. Then 
$$\sum_{z_1z_2z_3z_4\,:\,z_1=z} *\nabla_{\!Q} u(z_1z_2z_3z_4) \cdot\overrightarrow{z_4z_2}=
\sum_{z_1z_2z_3z_4\,:\,z_1=z} \nabla_{\!Q} v(z_1z_2z_3z_4) \cdot\overrightarrow{z_4z_2}=
\sum_{z_1z_2z_3z_4\,:\,z_1=z} (v(z_2)-v(z_4))=0
$$
because the diagonals of type $z_2z_4$ form a closed cycle around a nonboundary vertex $z$.

Let us prove the ``if'' part. Assume that $\sum_{z_1z_2z_3z_4\,:\,z_1=z} *\nabla_{\!Q} u(z_1z_2z_3z_4) \cdot\overrightarrow{z_4z_2}=0$ for each $z\in Q^0-\partial Q$. For an oriented edge $z_2z_4\subset W$ denote $V(z_2z_4):=*\nabla_{\!Q} u(z_1z_2z_3z_4) \cdot\overrightarrow{z_2z_4}$, where $z_1z_2z_3z_4$ is the face of $Q$ containing $z_2z_4$. Then
for each bounded face $w_1w_2\dots w_m$ of the graph $W$ we have
$\sum_{1\le k\le m} V(w_kw_{k+1})=0$, where $w_{m+1}:=w_1$.
Then a function $v\colon W^0\to\mathbb{R}$ is well-defined by the formula $v(w_m):=\sum_{1\le k<m}V(w_kw_{k+1})$, where $w_1w_2\dots w_m$ is a path in the graph $W$ joining $w_m$ with a fixed vertex $w_1$. 
Define a function $v\colon B^0\to\mathbb{R}$ analogously. Consider the combined function $v\colon Q^0\to\mathbb{R}$. Then for each face $z_1z_2z_3z_4$ of the lattice $Q$ we have $v(z_4)-v(z_2)=V(z_2z_4)=*\nabla_{\!Q} u \cdot\overrightarrow{z_2z_4}$ and analogously
$v(z_3)-v(z_1)=*\nabla_{\!Q} u \cdot\overrightarrow{z_1z_3}$.
Thus $\nabla_{\!Q} v=*\nabla_{\!Q} u$,
and $v$ is conjugate to $u$.
\end{proof}

\begin{lemma} \label{cl-laplacian-form}
For each vertex $z\in Q^0$ we have
$$[\Delta_Q u](z)=
\sum_{z_1z_2z_3z_4\,:\,z_1=z} *\nabla_{\!Q} u(z_1z_2z_3z_4)\cdot\overrightarrow{z_4z_2},$$
where the sum is over all the faces $z_1z_2z_3z_4$ of $Q$ such that $z_1=z$ with the vertices listed clockwise.
\end{lemma}

\begin{proof}
Let $v\colon Q^0\to\mathbb{R}$ be equal to $1$ at the vertex $z$ and $0$ at all the other vertices.
Differentiating the energy $E_Q(u)$ and applying
the identity
$(c\cdot d)(*a\cdot b)=(*c\cdot b)(a\cdot d)-(*c\cdot a)(b\cdot d)$ we get
\begin{align*}[\Delta_Q u](z)&=
\sum_{z_1z_2z_3z_4\subset Q}(\nabla_{\!Q} u \cdot\nabla_{\!Q} v) (-2\mathrm{Area}(z_1z_2z_3z_4))\\&=
\sum_{z_1z_2z_3z_4\subset Q}
(\nabla_{\!Q} u \cdot\nabla_{\!Q} v)
(*\overrightarrow{z_4z_2}\cdot\overrightarrow{z_1z_3})\\&=
\sum_{z_1z_2z_3z_4\subset Q}
\left((*\nabla_{\!Q} u \cdot \overrightarrow{z_1z_3})
(\nabla_{\!Q} v \cdot \overrightarrow{z_4z_2})-
(*\nabla_{\!Q} u \cdot \overrightarrow{z_4z_2})
(\nabla_{\!Q} v \cdot \overrightarrow{z_1z_3})
\right)\\&=
\sum_{z_1z_2z_3z_4\subset Q}
\left((*\nabla_{\!Q} u \cdot \overrightarrow{z_1z_3})
(v(z_2)-v(z_4))-
(*\nabla_{\!Q} u \cdot \overrightarrow{z_4z_2})
(v(z_3)-v(z_1))
\right)\\&=
\sum_{z_1z_2z_3z_4\,:\, z_1=z} *\nabla_{\!Q} u\cdot\overrightarrow{z_4z_2}.\\[-1.6cm] 
\end{align*}
\end{proof}

\begin{proof}[Proof of Variational Principle~\ref{cl-laplacian}]
By Lemmas~\ref{cl-analyticity}--\ref{cl-laplacian-form} a function is discrete harmonic if and only if
its laplacian vanishes at nonboundary vertices. By Convexity Principle~\ref{cl-min-energy} the latter is equivalent to having minimal energy among the functions with the same boundary values.
\end{proof}

\subsection{Existence and Uniqueness Theorem}

\begin{proof}[Proof of Existence and Uniqueness Theorem~\ref{cl-dirichlet-problem}]
By Convexity Principle~\ref{cl-min-energy} the energy $E\colon \mathbb{R}^{Q^0-\partial Q}\to\mathbb{R}$ has a unique global minimum $u\in \mathbb{R}^{Q^0-\partial Q}$. By Variational Principle~\ref{cl-laplacian} the function $u$ is the solution of the Dirichlet problem and it is unique.
\end{proof}

\begin{remark} \label{rem-Riemann}
Define a \emph{discrete Riemann surface} to be a
cell decomposition $Q$ of an orientable surface with quadrilateral faces together with an identification of each face with a quadrilateral $z_1z_2z_3z_4\subset\mathbb{C}$ by an orientation preserving homeomorphism. 
(No agreement of such identifications for different faces is assumed.)
The results of Sections~\ref{ssec-energy} and~\ref{sec-variational-principle} remain true for an arbitrary simply-connected discrete Riemann surface, not necessarily a quadrilateral lattice in the complex plane, if the faces are replaced by the corresponding quadrilaterals $z_1z_2z_3z_4\subset\mathbb{C}$ in all our constructions.
\end{remark}

\subsection{Maximum Principle}\label{ssec-maximum}

Let us discuss the case of orthogonal lattices in more detail. 
In this case $c(z_1z_3)=i\frac{z_2-z_4}{z_1-z_3}>0$, hence 
the value of a discrete harmonic function $u$ at a nonboundary vertex of $B$ equals to the weighted mean of the values at the neighbors. This immediately implies the following known result.

\begin{maximum-principle} \label{cl-maximum-principle} Let $Q$ be an orthogonal lattice and let $u\colon Q^0\to \mathbb{R}$ be a discrete harmonic function. Then
\vspace{-0.4cm}
$$
\max_{z\in Q^0}u(z)=\max_{w\in Q^0\cap \partial Q}u(w)
\quad\text{ and }\quad
\max_{z\in B^0}u(z)=\max_{w\in B^0\cap \partial Q}u(w).
$$
\end{maximum-principle}


For an orthogonal lattice Existence and Uniqueness Theorem~\ref{cl-dirichlet-problem} is an immediate corollary of Maximum Principle~\ref{cl-maximum-principle} and the finite-dimensional Fredholm alternative.
However, the principle does not hold for nonorthogonal lattices in general.

\begin{example}\label{ex-nomax} (S.~Tikhomirov)
Let $M>1$ and let $Q$ be the lattice formed by the quadrilateral with the vertices $0,\cot(\pi/8),\sqrt{2}M(\cot(\pi/8)+i),i$, and the $3$ other quadrilaterals obtained by symmetries with respect to the origin and the coordinate axes; see Figure~\ref{fig-no-max-pr}. Define a discrete analytic function $f\colon Q^0\to\mathbb{C}$ 
and a discrete harmonic function $u:=\Real f$
by the following table:
\begin{center}
\begin{tabular}{r|ccccc}
\hline
$z$ & $0$ & $\pm i$ & $\pm\cot\frac{\pi}{8}$ &
  $\pm \sqrt{2}M(\cot\frac{\pi}{8}+i)$ &
  $\pm \sqrt{2}M(\cot\frac{\pi}{8}-i)$
\\\hline
$f(z)$ & $M(1+i)$ & $1$ & $0$ & $0$ & $2Mi$
\\\hline
$u(z)$ & $M$ & $1$ & $0$ & $0$ & $0$
\\\hline
\end{tabular}
\end{center}
Then $\max_{Q^0}u/\max_{\partial Q}u=M$ can be arbitrarily large for large $M$. If $B$ is the graph formed by $4$ diagonals from the origin then $\max_{B^0}u>0$ whereas $\max_{B^0\cap\partial Q}u =0$.
\end{example}

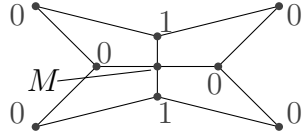
\begin{figure}[htbp]
\begin{center}
\input{complex-figure8v3.tex}
\caption{A discrete harmonic function on a nonorthogonal quadrilateral lattice not satisfying the maximum principle. The values of the function are shown near the vertices.}
\label{fig-no-max-pr}
\end{center}
\end{figure}

\subsection{Green's Identity}\label{ssec-harmonic}


Let us state one more result specific for orthogonal lattices, which is required for the sequel.


\begin{Green-identity}
\label{cl-green} Let $Q$ be an orthogonal lattice
and $u,v\colon B^0\to\mathbb{C}$ be arbitrary functions. Then
$$
\sum_{z\in 
B^0}
\left[ u\Delta_Q v - v\Delta_Q u  \right](z)=
0.
$$
\end{Green-identity}

\begin{proof} For an orthogonal lattice the energy splits as $E_Q(u)=E_B(u)+E_W(u)$, where $E_B(u)$ and $E_W(u)$ depend only on the values of the function $u$ at the vertices of $B$ and $W$, respectively. For an arbitrary homogeneous quadratic form $E_B(u)$ we have
$\sum_{z\in B^0}\left(u(z)\frac{\partial E_B(v)}{\partial v(z)}-v(z)\frac{\partial E_B(u)}{\partial u(z)}\right)=0$,
which is equivalent to the required identity.
\end{proof}


The identity contains no special boundary terms in contrast to \cite[Formula (2.4)]{Chelkak-Smirnov-08} (because our discrete laplacian  $\Delta_Q u$ is not equal to the one of \cite{Chelkak-Smirnov-08} at the boundary).
For nonorthogonal lattices Green's identity does not remain true unless one replaces summation over $B^0$ by summation over~$Q^0$.

\section{Convergence}\label{sec-scaling-limit}

In this section we prove Convergence Theorem~\ref{th-uniform-approx} and the results stated in Section~\ref{ssec-plan}. We also prove 
Friedrichs Inequality Lemma~\ref{cl-strip-estimate} below which perhaps may be useful for further investigations.


\subsection{Geometric preliminaries}\label{ssec-geometry}

Let us start with some basic estimates involving the lattice eccentricity. 
In what follows we use the following notation:
\begin{itemize}
\item $\mathrm{Const}$ is a positive constant (like $2\pi$ or $10^{100}$) that does not depend on any parameters of a configuration under consideration (e.g., $h$, $e$, the shape of $Q$);
$\mathrm{Const}$ may denote distinct constants at distinct places of the text (e.g., in the sides of a formula like $2\cdot\mathrm{Const}\le\mathrm{Const}$);
\item $\mathrm{Const}_{x,y,z}$ is a positive constant depending only on the parameters $x,y,z$;
\item $x_n\preceq y_n$ means that $\lim\sup (x_n-y_n) \le 0$;
\item $x_n\rightrightarrows x$ means that $x_n$ uniformly converges to $x$;
\item $E_{R}(u)$ is the sum~(\ref{eq-def-energy}) over all the faces $z_1z_2z_3z_4$ of $Q$ containing an edge of a subgraph $R\subset B$.
\end{itemize}


\begin{path-energy-estimate}\label{cl-obvious} For any path $w_0w_1\dots w_{m}\subset B$ we have
$$
E_{w_0\dots w_m}(u)\ge \mathrm{Const}\frac{(u(w_m)-u(w_0))^2}{me^2}.
$$
\end{path-energy-estimate}

\begin{proof}
Estimating $|\nabla_{\!Q} u|$ through its projection, using condition~(D) from Section~\ref{ssec-state}, 
the inequalities $e\ge 1$ and $\sin(1/e)\ge\mathrm{Const}/e$, 
and the Schwarz inequality we get
\begin{align*}
E_{w_0\dots w_m}(u)&=
\sum|\nabla_{\!Q} u|^2\mathrm{Area}(z_1z_2z_3z_4)\\
&\ge
\sum \frac{(u(z_1)-u(z_{3}))^2}{|z_1z_3|^2}\cdot\frac12
|z_1z_3|\,|z_2z_4|\sin\angle(z_1z_3,z_2z_4)\\
&\ge
\sum_{k=1}^{m}\frac{(u(w_k)-u(w_{k-1}))^2}{2e}\cdot\sin\frac{1}{e}\\
&\ge
\mathrm{Const}\frac{(u(w_m)-u(w_0))^2}{me^2},
\end{align*}
where the first two sums are over all faces $z_1z_2z_3z_4$ of $Q$ containing an edge of the path $w_0\dots w_m$.
\end{proof}

\begin{projection-lemma}\label{cl-projection}
For any face $z_1z_2z_3z_4$ of the lattice $Q$ and any vector $\overrightarrow v$ we have
$$
|\overrightarrow v|\le \mathrm{Const}\cdot e\cdot\left(\frac{|\overrightarrow v\cdot\overrightarrow{z_1z_3}|}{|{z_1z_3}|}+
\frac{|\overrightarrow v\cdot\overrightarrow{z_2z_4}|}{|{z_2z_4}|}\right).
$$
\end{projection-lemma}

\begin{proof}
Since $\angle(\overrightarrow{z_1z_3},\overrightarrow{z_2z_4})=
\pm\angle(\overrightarrow{v},\overrightarrow{z_1z_3})
\pm\angle(\overrightarrow{v},\overrightarrow{z_2z_4})$ by condition~(D) from Section~\ref{ssec-state}
it follows that at least one of the angles on the right-hand side, say, the first one, does not belong to the interval $(\frac{\pi}{2}-\frac{1}{2e},\frac{\pi}{2}+\frac{1}{2e})$.
Then
\vspace{-0.6cm}
$$
|\overrightarrow{v}|\le
\csc \frac{1}{2e}\frac{|\overrightarrow v\cdot\overrightarrow{z_1z_3}|}{|{z_1z_3}|}\le
\mathrm{Const}\cdot e\cdot \frac{|\overrightarrow v\cdot\overrightarrow{z_1z_3}|}{|{z_1z_3}|}. \vspace{-0.8cm}
$$
\end{proof}

\begin{lemma-on-vertices-in-rectangle}\label{cl-points} 
A rectangle $r\times h$ with the side $r>h$ contains at most $\mathrm{Const}\cdot er/h$ vertices of the graph~$B$.
\end{lemma-on-vertices-in-rectangle}

\begin{proof} The rectangle $r\times h$ can be covered by $\mathrm{Const}\cdot[r/h]$ discs of radius $h/2$. 
Then by condition~(U) from Section~\ref{ssec-state}
the number of vertices in the rectangle is less than $\mathrm{Const}\cdot er/h$. 
\end{proof}

\begin{diameter-lemma}\label{cl-diam} The diameter of each bounded face of the graphs $B$ and $W$ is at most $h$.
\end{diameter-lemma}

\begin{proof} A bounded face of the graph $B$ contains a vertex of the graph $W$. The vertex is joined by edges of the graph $Q$ with all the vertices of the face (and by ``half-edges'' of the graph $W$ with the break points of the $2$-segment edges of $B$ in nonconvex faces of $Q$). Since the edges of $Q$ have length at most $h/2$ it follows that the diameter of the face of $B$ is at most $h$.
\end{proof}

\subsection{Convergence of energy} \label{ssec-energy-bound}

For the proof of Energy Convergence Lemma~\ref{l-energy-convergence} we need the following lemma.

\begin{gradient-approximation-lemma}\label{cl-maximal-size}
We have $|\nabla g-\nabla_{\!Q} (g\left|_{Q^0}\right.)|\le \mathrm{Const}\cdot eh\max_{z\in\mathrm{Conv}(\partial Q)}|D^2 g(z)|$.
\end{gradient-approximation-lemma}


\begin{proof}
Consider a face $z_1z_2z_3z_4$ of the lattice $Q$. By the Rolle theorem there is a point $z\in z_1z_3$ (possibly outside $z_1z_2z_3z_4$ but inside the convex hull $\mathrm{Conv}(\partial Q)$) such that
$\left(\nabla g(z)-[\nabla_{\!Q} g](z_1z_2z_3z_4)\right)\cdot \overrightarrow{z_1z_3}/|\overrightarrow{z_1z_3}|=0$.
Thus $(\nabla g-\nabla_{\!Q} g)\cdot \overrightarrow{z_1z_3}/|\overrightarrow{z_1z_3}|\le \mathrm{Const}\cdot h \max_{z\in\mathrm{Conv}(z_1z_2z_3z_4)}|D^2 g(z)|$ in 
$z_1z_2z_3z_4$.
The same inequality holds with $z_1z_3$ replaced by $z_2z_4$.
By Projection Lemma~\ref{cl-projection} the lemma follows.
\end{proof}





\begin{proof}[Proof of Energy Convergence Lemma~\ref{l-energy-convergence}]
Denote by $\widehat{Q}_n$ the domain enclosed by the curve $\partial Q_n$. Since $Q_n$ approximates $\Omega$ and $\partial\Omega$ is smooth it follows that some neighborhood $\Omega'$ of $\Omega$ contains all the lattices $Q_n$ and $\mathrm{Area}(\Omega-\widehat{Q}_n),
\mathrm{Area}(\widehat{Q}_n-\Omega)\to 0$ as $n\to \infty$. Since the domain $\Omega$ is bounded and the function $g\colon\mathbb{C}\to\mathbb{R}$ is smooth 
it follows that $\nabla g$ is bounded in $\mathrm{Conv}(\Omega')$. Thus the integrals $E_{\Omega}(g)$, $E_{\widehat{Q}_n}(g)$ exist and $E_{\Omega}(g)-E_{\widehat{Q}_n}(g)=
E_{\Omega-\widehat{Q}_n}(g)-E_{\widehat{Q}_n-\Omega}(g)\to 0$ as $n\to \infty$. By Gradient Approximation Lemma~\ref{cl-maximal-size} we get
$E_{\widehat{Q}_n}(g)-E_{Q_n}(g\left|_{Q^0_n}\right.)\to 0$ as $n\to \infty$, and the lemma follows.
\end{proof}

\begin{remark}\label{rem-rough-boundary} The same argument proves the estimate $E_{Q_n}(g\left|_{Q^0_n}\right.)\preceq E_{\mathrm{Cl}\,\Omega}(g)$, if no smoothness assumptions are imposed on the boundary $\partial\Omega$ (the ``inequality'' can be strict for $\mathrm{Area}(\partial \Omega)>0$). 
\end{remark}

\begin{remark} For a harmonic function $v\colon \Omega\to\mathbb{R}$ and a sequence of functions $v_n\colon\Omega\to\mathbb{R}$ conditions $v_n=v$ on $\partial \Omega$ and $E_\Omega (v_n)\to E_\Omega (v)$ do not necessarily imply that $v_n\to v$ pointwise. For instance, take a continuous function $v_n\colon\mathbb{D}^2\to\mathbb{R}$ such that $v_n(z)=0$ for $|z|=1$, $v_n(z)=1$ for $|z|\le 1/n$, and $v_n(z)$ is harmonic in the ring $1/n<|z|<1$. Then $E_{\mathbb{D}^2}(v_n)=1/\ln n\to 0$ as $n\to\infty$ but $v_n(0)=1\not\to 0=v(0)$.
\end{remark}



\subsection{Equicontinuity}

First we prove Equicontinuity Lemma~\ref{l-equicontinuity} for a square lattice and then for the general orthogonal one. The latter proof is essentially the same but requires more technical details.

\begin{proof}[Proof of Equicontinuity Lemma~\ref{l-equicontinuity} for square lattices] First assume that $B$ (and also $Q$)
is a part of a square lattice and the segment joining $z$ and $w$ is contained in the graph $B$; see Figure~\ref{fig-rectangles}. 
For now denote by $h$ the step of the square lattice $B$. 
Let $R_m$ be the rectangle $2mh\times (2mh+|z-w|)$ centered at the point $(z+w)/2$ with the side $2mh$ orthogonal to the segment $zw$; see Figure~\ref{fig-rectangles}. Assume further that $m\in\mathbb{Z}$ and $1\le m\le\frac{r-|z-w|}{2h}$ so that $R_m\subset R$.
Denote $\partial B:=\partial Q\cap B^0$ and
$$
\delta:=\left|u(z)-u(w)\right|-\max_{z',w'\in R\cap \partial B}\left|u(z')-u(w')\right|.
$$
The required estimate~\eqref{eq-l-equicontinuity} holds automatically for $\delta\le 0$. Assume further that $\delta>0$. 

\begin{figure}[htbp]
\begin{center}
\input{complex-figure4v4.tex}
\caption{The rectangles $R_m$ and the points $z_m,w_m$ from the proof of Lemma~\ref{l-equicontinuity}.}
\label{fig-rectangles}
\end{center}
\end{figure}
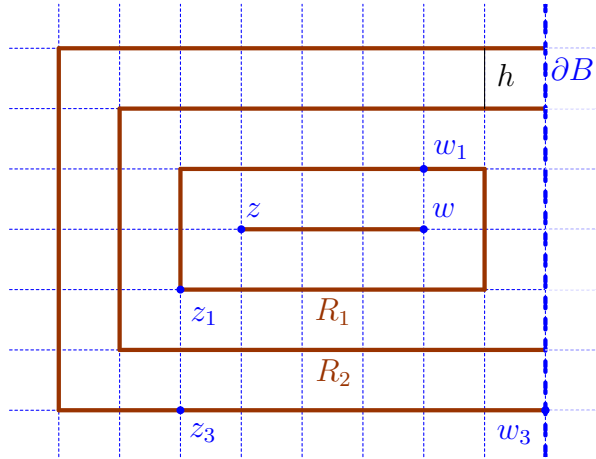

Let us prove that the graph $\partial R_m \cap B$
has two 
vertices $z_m,w_m$ joined by a path in $\partial R_m \cap B$ such that 
$\left|u(z_m)-u(w_m)\right|\ge \delta/2>0$.
Indeed, assume without loss of generality that $u(z)\ge u(w)$.  Consider the following two cases: (1) $R_m \cap \partial B=\emptyset$; and (2) $R_m \cap \partial B\ne\emptyset$. 

\emph{Case} (1). Since $R_m \cap \partial B=\emptyset$ it follows by Maximum Principle~\ref{cl-maximum-principle} that
the function $u$ attains its maximum and minimum over $R_m\cap B^0$ at some vertices $z_m,w_m\in \partial R_m\cap B^0$, respectively. Clearly, the vertices are joined by a path in the graph $\partial R_m \cap B$. Since $u(z_m)\ge u(z)\ge u(w)\ge u(w_m)$, the required inequality follows.

\emph{Case} (2). By Maximum Principle~\ref{cl-maximum-principle}
the function $u$ attains its maximum and minimum over $R_m\cap B^0$ 
at some vertices $z',w'\in (\partial R_m\cap B^0)\cup (R_m\cap \partial B)$, respectively. Without loss of generality we have the following $3$ possibilities:  
\begin{itemize}
\item[(i)] $z',w'\in \partial R_m\cap B^0$; 
\item[(ii)] $z'\in\partial R_m\cap B^0$, $w'\in R_m\cap \partial B$;
\item[(iii)] $z',w'\in R_m\cap \partial B$.
\end{itemize}
In subcase (i) since both $R_m \cap \partial B\ne\emptyset$ and $\partial R_m \cap B^0\ne\emptyset$ it follows that $\partial R_m \cap\partial B\ne\emptyset$. Join the vertices $z',w'$ with some boundary vertices $z'',w''\in \partial R_m \cap \partial B$ by two paths in the graph $\partial R_m \cap B$. Then we have $u(z')-u(z'')+u(w'')-u(w')\ge u(z)-u(w)-(u(z'')-u(w''))\ge \delta$. Hence without loss of generality $u(z')-u(z'')\ge \delta/2$, and $z_m:=z'$, $w_m:=z''$ are the required vertices. In subcase (ii) join the vertex $z_m:=z'\in\partial R_m\cap B$ with a boundary vertex $w_m\in \partial R_m \cap \partial B$ by a path in the graph $\partial R_m \cap B$. Then $u(z_m)-u(w_m)\ge u(z)-u(w)-(u(w_m)-u(w'))\ge \delta$, and the required inequality follows. In subcase (iii) we get 
$0\ge u(z)-u(w) - (u(z')-u(w'))\ge \delta$ which contradicts to the assumption $\delta >0$; thus this subcase is impossible. So we have found the required pair of points $z_m,w_m$ in each possible case.

Let us estimate the energy of the graph $\partial R_m\cap B$.
The length of the minimal path joining the points $z_m$ and $w_m$ 
is at most the number of vertices in the graph $\partial R_m\cap B$, which is at most $8m+2|z-w|/h$.
By Path Energy Lemma~\ref{cl-obvious} we get 
$$
E_{\partial R_m\cap B}(u) \ge \mathrm{Const}\frac{(u(z_m)-u(w_m))^2}{e^2(8m+2|z-w|/h)}
    \ge \mathrm{Const}\frac{\delta^2}{e^2}\cdot \frac{h}{mh+|z-w|/4}.
$$
Summing these inequalities for $m$ from $1$ to $\left[\frac{r-|z-w|}{2h}\right]$ and estimating the sum via an integral we get
\begin{multline*}
E_Q(u)\ge \sum_{m=1}^{\left[\frac{r-|z-w|}{2h}\right]}E_{\partial R_m\cap B}(u)
    \ge 
     \mathrm{Const}\frac{\delta^2}{e^2}\int\limits_{h}^{\frac{r-|z-w|}{2}}
     \frac{dt}{t+|z-w|/4}\\
    \ge
    \mathrm{Const}\frac{\delta^2}{e^2}\ln
    \frac{2r-|z-w|}{4h+|z-w|}
    \ge
    \mathrm{Const}\frac{\delta^2}{e^2}\ln
    \frac{r}{3\max\{|z-w|,h\}}.
\end{multline*}
This is equivalent to the required inequality.
\end{proof}

\begin{proof}[Proof of Equicontinuity Lemma~\ref{l-equicontinuity} in the general case] 
Consider an auxiliary square lattice
with edges parallel and orthogonal to $zw$
and with the step $h$ equal to twice the maximal edge length of the lattice ~$Q$. 
Define the rectangles $R_m$ literally as above. Let $F$ be the outer face of the lattice~$Q$.

The boundaries of the rectangles $R_m$ and $R_{m-1}$ are separated by a simple closed broken line lying in the set $B\cup F$.
Indeed, otherwise $\partial R_m$ and $\partial R_{m-1}$ would be joined by a path in the complement to the set $B\cup F$ and thus the graph $B$ would have a bounded face of diameter $>h$, which contradicts to Diameter Lemma~\ref{cl-diam}. Denote by $\tilde R_m$ the polygon bounded by our broken line.


The number of edges in the graph $\partial\tilde R_m\cap B$ is not greater than the number of vertices of the graph $B$ lying in the strip between $\partial R_m$ and $\partial R_{m-1}$. Thus by Rectangle Capacity Lemma~\ref{cl-points} the graph $\partial\tilde R_m\cap B$ contains less than $\mathrm{Const}\cdot e(4m+|z-w|/h)$ edges.
Now the same energy estimates as for the square lattice (with the rectangle $R_m$ replaced by the polygon $\tilde R_m$)
prove the required inequality.
\end{proof}

\begin{remark}\label{rem-relaxed}
Equicontinuity Lemma~\ref{l-equicontinuity} and its proof remain true if condition~(D) in the definition of the eccentricity $e$ is relaxed to the following one:
\begin{itemize}
\item[(D')] for each pair of intersecting edges $z_1z_3\subset B$ and $z_2z_4\subset W$ we have $|z_1z_3|/|z_2z_4|<\mathrm{Const}$ and  $1/\mathrm{Const}<\angle (z_1z_3,z_2z_4)<\pi - 1/\mathrm{Const}$ (no lower bound for the ratio $|z_1z_3|/|z_2z_4|$ is assumed). 
\end{itemize}
Indeed, condition (D) was used only in the proof of Path Energy Lemma~\ref{cl-obvious}, which remains true with relaxed condition~(D'). 
\end{remark}

\begin{remark}\label{rem-close-points}
Equicontinuity Lemma~\ref{l-equicontinuity} and its proof remain true 
for 
a discrete Riemann surface $Q\subset\mathbb{R}^2$ (in the sense of Remark~\ref{rem-Riemann}) with all the quadrilaterals $z_1z_2z_3z_4$ having orthogonal diagonals.
\end{remark}

\begin{remark} 
Our proof of Equicontinuity Lemma~\ref{l-equicontinuity} cannot be generalized to nonorthogonal lattices because
it essentially uses
Maximum Principle~\ref{cl-maximum-principle} not holding for them. 
\end{remark}

\subsection{Approximation of laplacian}\label{ssec-precompactness}

First we prove Laplacian Approximation Lemma~\ref{cl-approximation} in several particular cases and then combine them together. Throughout this subsection $z$ denotes the coordinate in the complex plane $\mathbb{C}$, so that, e.g., $\Real z$ denotes the function $z\mapsto \Real z$.

\begin{proof}[Proof of Lemma~\ref{cl-approximation} for $g(z)=1$]
By Lemma~\ref{cl-laplacian-form} the laplacian of a constant function vanishes, and the lemma follows.
\end{proof}

\begin{proof}[Proof of Lemma~\ref{cl-approximation} for $g(z)=\Real z$]
Clearly, the function
$f(z)=z$ is discrete analytic. Thus $g(z)=\Real z$ is discrete harmonic. Then by Variational Principle~\ref{cl-laplacian} it follows that $[\Delta_Q\Real z](w)=0$ for each vertex $w\in Q^0-\partial Q$. On the other hand, $\Delta \Real z=0$ as well, and the lemma follows.
\end{proof}

\begin{proof}[Proof of Lemma~\ref{cl-approximation} for $g(z)=\Imaginary z$]
This is proved by the previous argument with $iz$ instead of~$z$.
\end{proof}

For the next case of the lemma we need the following lemma.

\begin{lemma}\label{cl-last} The number of faces $z_1z_2z_3z_4$ such that $z_1\in B^0\cap R$, $z_3\not\in R$ is at most $\mathrm{Const}\cdot er/h$.
\end{lemma}

\begin{proof}
By Diameter Lemma~\ref{cl-diam} it follows that the vertices of such faces are contained in the $h$-neighborhood of $\partial R$. By Rectangle Capacity Lemma~\ref{cl-points} the number of these vertices is at most $\mathrm{Const}\cdot er/h$. Then by the Euler formula for planar graphs it follows that the number of faces is also at most $\mathrm{Const}\cdot er/h$.
\end{proof}

\begin{proof}[Proof of Lemma~\ref{cl-approximation} for $g(z)=|z|^2$]
For a face $z_1z_2z_3z_4$ of the lattice $Q$ denote by $z'$ the intersection point of the bisectors of the diagonals $z_1z_3$ and $z_2z_4$. Clearly, then $\nabla_{\!Q}|z-z'|^2(z_1z_2z_3z_4)=0$.
Using the two previous cases of the lemma and Lemmas~\ref{cl-analyticity}, \ref{cl-laplacian-form} we get
\begin{align*}
[\Delta_Q|z|^2](w)
&=[\Delta_Q|z-w|^2](w)
+[\Delta_Q(2\Real \bar wz-|w|^2)](w)\\
&=\sum_{z_1z_2z_3z_4\,:\,z_1=w}
(*\nabla_{\!Q}|z-w|^2)\cdot \overrightarrow{z_4z_2}\\
&=\sum_{z_1z_2z_3z_4\,:\,z_1=w}\left(
2*\!\nabla_{\!Q}\Real(\overline{(z'-w)}z)+ *\nabla_{\!Q}|z-z'|^2\right)\cdot \overrightarrow{z_4z_2}\\
&=\sum_{z_1z_2z_3z_4\,:\,z_1=w}
2\nabla_{\!Q}\Imaginary(\overline{(z'-w)}z)\cdot \overrightarrow{z_4z_2}\\
&=\sum_{z_1z_2z_3z_4\,:\,z_1=w}
2\Imaginary\left(\overline{(z'-w)}(z_2-z_4)\right)\\
&=\sum_{z_1z_2z_3z_4\,:\,z_1=w}4\mathrm{Area}(wz_2z'z_4).
\end{align*}
Here and in the next paragraph the area of a closed broken line is understood in oriented sense.

Let us estimate the difference between the sum of such areas and the area of the square $R$. Assume for simplicity that all the faces of $Q$ are convex (otherwise the diagonals $z_2z_4$ in the argument below should be replaced by edges $z_2z_4\subset W$ which could be $2$-segment broken lines). First, denote by $R_Q$ the union of the oriented diagonals $\overrightarrow{z_2z_4}$ of all the faces $z_1z_2z_3z_4$ of $Q$ such that $z_1\in R\cap B^0$ and $z_3\not\in R$. 
It is the oriented boundary of the union of all the triangles of the form $z_1z_2z_4$ with $z_1\in R\cap B^0$. Since $R_Q$ is  contained in the $h/2$-neighborhood of the curve $\partial R$ it follows that $|\mathrm{Area}(R_Q)-\mathrm{Area}(R)|<4hr$.
Second, by sine theorem for the triangle with the vertices $z'$,$\frac{z_1+z_3}{2}$, $\frac{z_2+z_4}{2}$ and condition~(D) from Section~\ref{ssec-state} it follows that
$$
\left|z'-\frac{z_2+z_4}{2}\right|\le 
\csc\angle z'\cdot\left|\frac{z_1+z_3}{2}-\frac{z_2+z_4}{2}\right|=
\csc\angle (z_1z_3,z_2z_4)\cdot\left|\frac{z_1-z_2}{2}+\frac{z_3-z_4}{2}\right|\le
\mathrm{Const}\cdot eh,
$$
thus $|\mathrm{Area}(z'z_2z_4)|\le  \mathrm{Const}\cdot eh^2$.  
(If the triangle in question is degenerate then the same estimate follows by a limiting argument.)
Summing up the above expressions for the laplacian over all the vertices $w\in R\cap B^0$, applying the obtained area estimates and Lemma~\ref{cl-last} we get
\begin{equation*}
\left
|
\sum_{w\in R\cap B^0}[\Delta_Q |z|^2](w)-
\int_R \Delta |z|^2 dxdy
\right
|
=
\left
|
4\sum_{z_1z_3\pitchfork\partial R}\mathrm{Area}(z'z_2z_4)+
4\mathrm{Area}(R_Q)-4\mathrm{Area}(R)
\right
|
\le
\mathrm{Const}\cdot e^2hr.
\end{equation*}
The second sum is over all the faces $z_1z_2z_3z_4$ of the graph $Q$ such that $z_1\in R\cap B^0$ and $z_3\not\in R$.
\end{proof}



\begin{proof}[Proof of Lemma~\ref{cl-approximation} for $g(z)=\Real z^2$] 
For a face $z_1z_2z_3z_4$ of the lattice $Q$ denote by
$z''$ the point such that 
$\mathrm{Re}((z_1-z_3)(z''-\frac{z_1+z_3}{2}))=
\mathrm{Re}((z_2-z_4)(z''-\frac{z_2+z_4}{2}))=0
$. 
Then $[\nabla_{\!Q}\mathrm{Re}(z-z'')^2](z_1z_2z_3z_4)=0$.
Analogously to the previous case of the lemma we get
$$
[\Delta_Q\Real z^2](w)
=\sum_{z_1z_2z_3z_4\,:\,z_1=w}
2\Imaginary\left((z''-w)(z_2-z_4)\right).
$$
Summing up this expressions over all $w\in R\cap B^0$ and canceling repeating terms we get
$$
\sum_{w\in R\cap B^0}[\Delta_Q \Real z^2](w)=
\sum_{z_1z_3\pitchfork\partial R}2\Imaginary\left(z''(z_2-z_4)\right)=
\sum_{z_1z_3\pitchfork\partial R} \Imaginary\left((2z''-z_2-z_4)(z_2-z_4)\right).
$$
Here the first sum is over all the vertices $w\in R\cap B^0$ and
the other sums are over the faces $z_1z_2z_3z_4$ of $Q$ such that $z_1\in B^0\cap R$ and $z_3\not\in R$. The last equality holds because the diagonals $\overrightarrow{z_2z_4}$ of these faces form
the oriented boundary of a finite union of triangles.

Analogously to the previous case of the lemma we get $|z''-\frac{z_2+z_4}{2}|\le \mathrm{Const}\cdot eh$. Thus by Lemma~\ref{cl-last}
$$
\left
|
\sum_{w\in R\cap B^0}\left[\Delta_Q \Real z^2\right](w)
- \int_R \Delta \Real z^2\, dxdy
\right
|
\le \sum_{z_1z_3\pitchfork\partial R} |2z''-z_2-z_4|\cdot|z_2-z_4|\le 
\mathrm{Const}\cdot e^2hr.\\[-0.7cm]
$$
\end{proof}

\begin{proof}[Proof of Lemma~\ref{cl-approximation} for $g(z)=\Imaginary z^2$] This is analogous to the previous case.
\end{proof}

\begin{proof}[Proof of Lemma~\ref{cl-approximation}
in the case when $D^k g=0$ at the center of $R$ for $k=0,1,2$]
Using the estimate $|\Delta g(z)|\le \mathrm{Const}\cdot r\max_{z\in R}|D^3 g(z)|$ we get
$$
\left|\int_R \Delta g \,dxdy\right|\le \mathrm{Const}\cdot r^3\max_{z\in R}|D^3 g(z)|.
$$
Now applying
Lemma~\ref{cl-laplacian-form}, canceling repeating terms, applying Lemma~\ref{cl-last}, Gradient Approximation Lemma~\ref{cl-maximal-size}, the estimates $|\nabla g(z)|\le \mathrm{Const}\cdot r^2\max_{z\in R}|D^3 g(z)|$, and $|D^2 g(z)|\le \mathrm{Const}\cdot r\max_{z\in R}|D^3 g(z)|$
we get
\begin{align*}
\left|
\sum_{w\in R\cap B^0}[\Delta_Q g](w)
\right|
&=
\left|
\sum_{w\in R\cap B^0}\quad
\sum_{z_1z_2z_3z_4\,:\,z_1=w} *\nabla_{\!Q} g\cdot\overrightarrow{z_4z_2}
\right|\\
&=
\left|
\sum_{z_1z_3\pitchfork\partial R}*\nabla_{\!Q} g\cdot\overrightarrow{z_4z_2}
\right|\\
&\le
\mathrm{Const}\cdot er/h\cdot (r^2+ehr)\max_{z\in R}|D^3 g(z)|\cdot h\\
&\le
\mathrm{Const}\cdot e^2r^3\max_{z\in R}|D^3 g(z)|.\\[-1.6cm]
\end{align*}
\end{proof}

\begin{proof}[Proof of Lemma~\ref{cl-approximation} in the general case]
Assume without loss of generality that the center of the square $R$ is the origin. By the Taylor formula
$$
g(z)=
a_0+a_1\Real z+a_2\Imaginary z +a_3\Real z^2+a_4 |z|^2+a_5\Imaginary z^2+
\tilde g(z),
$$
for some $a_0,\dots,a_5\in\mathbb{R}$ and
a function $\tilde g(z)$
such that $D^k\tilde g(0)=0$ for $k=0,1,2$.
For each of the summands the lemma has already been proved. Thus the general case follows.
\end{proof}


\begin{remark}\label{rem-laplace-approx} A similar argument proves the estimate 
$$\left|
\left[\Delta_Q (g\left|_{Q^0}\right.)\right](z)
\right|\le
\mathrm{Const}_e\left(h^2\max_{w:|w-z|\le h}|D^2 g(w)|+h^3\max_{w:|w-z|\le h}|D^3 g(w)|\right).
$$ 
\end{remark}

\subsection{Uniform limit}\label{ssec-uniform}

The following lemma is the last result we need before the proof of Convergence Theorem~\ref{th-uniform-approx}. 





\begin{limit-harmonicity-lemma} \label{cl-harmonic-limit}
Let $\{Q_n\}$ be a nondegenerate uniform sequence of orthogonal lattices approximating a domain $\Omega$.
Let $u_n\colon Q^0_n\to\mathbb{R}$ be a sequence of discrete harmonic functions.
Suppose that the restrictions $u_n\left|_{B^0_n}\right.$
uniformly converge to a continuous function $u\colon \Omega\to\mathbb{R}$;  
 then the function $u$ is harmonic.
\end{limit-harmonicity-lemma}

\begin{proof}[Proof of Lemma~\ref{cl-harmonic-limit}]
Take an arbitrary smooth function $v\colon \Omega\to\mathbb{R}$ vanishing outside a compact subset $K\subset\Omega$ with smooth boundary.
By the Weyl lemma it suffices to prove that
$\int_{\Omega}u\Delta v \, dxdy=0$. 

Let us estimate the difference between $\int_{\Omega}u\Delta v\, dxdy$ and its discrete counterpart. For each $n$ take an auxiliary infinite square lattice with edge length $r:=\sqrt{h}$. For a face $R$ of the $n$-th auxiliary lattice denote $\tilde u_n(R):=\max_{z\in R\cap K}u(z)$.
Then $\tilde u_n\rightrightarrows u$ on the compact set $K$ because $u$ is continuous. 
Applying the convergence $u_n,\tilde u_n\rightrightarrows u$, the boundness of $\Delta v$, Remark~\ref{rem-laplace-approx}, condition~(U) from Section~\ref{ssec-state},
and then Laplacian Approximation Lemma~\ref{cl-approximation} we get
\begin{multline*}
\left|\int_\Omega u\Delta v \, dxdy - \sum_{z\in B_n^0} [u_n\Delta_{Q_n} v](z)\right|
\preceq
\sum_{R\,:\,R\cap K\ne\emptyset} |\tilde u_n(R)|\cdot
\left|
\int_R \Delta v\, dxdy -
\sum_{z\in R\cap B^0_n}[\Delta_{Q_n} v](z)
\right|\preceq\\
\mathrm{Area}(K)/r^2\cdot \max_K|u|\cdot \mathrm{Const}_e\cdot (rh+r^3)\max_K|D^2v,D^3v|\le \mathrm{Const}_{e,K,u,v}\cdot\sqrt{h}\to 0
\quad\text{ as }\quad n\to\infty.
\end{multline*}

It remains to estimate the discrete counterpart of $\int_{\Omega}u\Delta v \, dxdy$. Take $n$ large enough so that $K$ is inside $\partial Q_n$. Applying Green's Identity~\ref{cl-green} 
and the assumptions that $u_n$ is discrete harmonic
and $v$ vanishes outside $K$, we get
$$
\sum_{z\in B_n^0} [u_n\Delta_{Q_n} v](z)=\sum_{z\in B_n^0} [v \Delta_{Q_n} u_n](z)=0.
$$
Thus $\int_{\Omega}u\Delta v \, dxdy=0$, which proves the lemma.
\end{proof}

Lemma~\ref{cl-harmonic-limit} and its proof remain true, if condition (D) from Section~\ref{ssec-state} is relaxed to condition (D') from Remark~\ref{rem-relaxed}.

\subsection{Convergence Theorem}\label{ssec-convergence}

\begin{proof}[Proof of Convergence Theorem~\ref{th-uniform-approx}]
Take an arbitrary subsequence $Q_{n_k}$ of the given sequence of lattices $Q_n$. For brevity denote $Q_k:=Q_{n_k}$, $B_k:=B_{n_k}$. 

Let us estimate $|u_{Q_k,g}|$.
Since the sequence $Q_k$ approximates the bounded domain $\Omega$ it follows that there is a disk $\Omega'$ containing all lattices $Q_k$.
By Maximum Principle~\ref{cl-maximum-principle} we have
$$\max_{z\in B_k^0}|u_{Q_k,g}(z)|=
\max_{w\in B_k^0\cap \partial Q_k}|u_{Q_k,g}(w)|\le \max_{w\in \Closure\Omega'}|g(w)|<\infty$$
because $u_{Q_k,g}=g$ at $\partial Q_k$ and $g\colon\mathbb{C}\to \mathbb{R}$ is continuous. So the sequence $u_{Q_k,g}$ is uniformly bounded.

Let us estimate the right-hand side of inequality~\eqref{eq-l-equicontinuity} from Equicontinuity Lemma~\ref{l-equicontinuity} for $u:=u_{Q_k,g}$, $r:=3\mathrm{Diam}(\Omega')^{1/2}|z-w|^{1/2}$, and $|z-w|\ge h$.
By Variational Principle~\ref{cl-laplacian}, Energy Convergence Lemma~\ref{l-energy-convergence} and Remark~\ref{rem-rough-boundary} we have
$E_{Q_k}(u_{Q_k,g})\le E_{Q_k}(g\left|_{Q^0_k}\right.)\preceq E_{\mathrm{Cl}\Omega}(g)<\infty$ for $k\to\infty$. Thus the sequence $E_{Q_k}(u_{Q_k,g})$ is bounded. 
The second summand in~\eqref{eq-l-equicontinuity} is bounded by 
$2r\cdot\max_{z\in\Omega'}|D^1g(z)|$.
Thus by Equicontinuity Lemma~\ref{l-equicontinuity} there is a constant $\mathrm{Const}_{\{Q_n\},\Omega,g}$ 
such that
$$|u_{Q_k,g}(z)-u_{Q_k,g}(w)|\le\mathrm{Const}_{\{Q_n\},\Omega,g}\left(\ln^{-1/2}\frac{\mathrm{Diam}(\Omega')}{|z-w|}+|z-w|^{1/2}\right)\to 0\text{ as }|z-w|\to 0.$$
For $|z-w|< h$ the same estimate holds with $|z-w|$ replaced by $h$.
Thus still $|u_{Q_k,g}(z)-u_{Q_k,g}(w)|\to 0$ as $|z-w|\to 0$
because
 there are only finitely many lattices $Q_k$ with maximal edge length greater than a given number. 


We have proved that $\left\{u_{Q_k,g}\left|_{B_k^0}\right.\right\}$ is \emph{equicontinuous},
i.~e., there is a positive function $\delta(\epsilon)$ not depending on $k$ such that for each $z,w\in B^0_k$ with $|z-w|<\delta(\epsilon)$ we have $|u_{Q_k,g}(z)-u_{Q_k,g}(w)|<\epsilon$. 
By the Arzel\`a-Ascoli theorem it follows that 
a subsequence of the sequence $u_{Q_k,g}\left|_{B_k^0}\right.$
 uniformly converges to a function $u\colon \mathrm{Cl}\Omega \to \mathbb{R}$ 
continuous in the \emph{closure} of the domain $\Omega$.



Each boundary point $z\in\partial\Omega$ is a limit of a sequence $z_k\in\partial Q_k\cap B_k^0$, hence 
$u(z)=\lim u_{Q_k,g}(z_k)=\lim g(z_k)=g(z)$, where
the functions $u_{Q_k,g}$ are taken from our converging subsequence. By Lemma~\ref{cl-harmonic-limit} the function $u$ is harmonic in $\Omega$. Thus the function $u$ is the solution $u_{\Omega,g}$ of the Dirichlet problem on $\Omega$.
Since the solution $u_{\Omega,g}$ is unique it follows that the initial sequence $u_{Q_n,g}\colon B^0_n\to\mathbb{R}$ uniformly converges to $u_{\Omega,g}$. Analogously, $u_{Q_n,g}\colon W^0_n\to\mathbb{R}$ uniformly converges to $u_{\Omega,g}$. This completes the proof of main results.
\end{proof}


\begin{remark}\label{rem-relaxed-convergence}
Convergence Theorem~\ref{th-uniform-approx} and its proof 
remain true, if condition~(D) from Section~\ref{ssec-state} is relaxed to condition~(D') from Remark~\ref{rem-relaxed}, at the cost of providing the uniform convergence only at the vertices of the graphs~$B_n$ but not $W_n$.
\end{remark}

\begin{remark} Convergence Theorem~\ref{th-uniform-approx} and its proof 
remain true, if instead of orthogonality of the lattices we assume that the oriented angle between the diagonals $z_1z_3\subset B$ and $z_2z_4\subset W$ is the same for all faces $z_1z_2z_3z_4$ of each lattice. Indeed, the orthogonality was used only in the proof of Maximum Principle~\ref{cl-maximum-principle} and Green's Ideentity~\ref{cl-green}, and it is easy to check that these assertions remain true for the more general class of lattices in question (provided that $v(\partial Q)=0$ in Green's Identity~\ref{cl-green}). 
\end{remark}


\subsection{The Friedrichs inequality} \label{ssec-boundary-condition}


Let us conclude this section by one more energy estimate not used in the proof of main results. 
For a subset $K\subset \mathbb{C}$ and a function $u\colon B^0\to \mathbb{R}$ denote $L^{\!2}_K(u):=\sum_{z\in K\cap B^0} u^2(z)$. 


\begin{friedrichs-lemma}\label{cl-strip-estimate}
Let $\Omega$ be a bounded simply-connected domain with smooth boundary,
$Q$ be a quadrilateral lattice, 
$K$ be a compact set inside $\partial Q$, and $u\colon B^0\to \mathbb{R}$ be an arbitrary function.
Denote $r:=\max_{z\in K}\mathrm{Dist}(z,\partial \Omega)$.
Assume that $r>h$ and
$r>\max_{z\in \partial Q}\mathrm{Dist}(z,\partial \Omega)$. Then there is a constant $\mathrm{Const}_{\Omega,e}$ depending only on $\Omega$ and $e$ (but not on $Q$, $K$, $r$, $h$, $u$) such that
$$
h^2L^{\!2}_{\Omega-K}(u)\le
\mathrm{Const}_{\Omega,e} \left(hrL^{\!2}_{\partial Q}(u)+r^2 E_Q(u)\right).
$$
\end{friedrichs-lemma}



For the proof of Friedrichs Inequality Lemma~\ref{cl-strip-estimate} we need some notation and auxiliary lemmas. Let $abcd$ be a rectangle with vertices $a,b,c,d\in\mathbb{C}$ listed clockwise such that $ab$ is outside $\partial Q$ and $|b-c|>h$; see Figure~\ref{fig-strip}. Denote $r:=|b-c|$. For a point $z\in abcd$ denote by $R_z$ the rectangle $(r+2h)\times 2h$ centered at $(z'+w')/2$ with the side $2h$ parallel to $ab$, where
$z'\in cd$ and $w'\in ab$ are the points such that $z\in z'w'$ and $z'w'\parallel bc$.

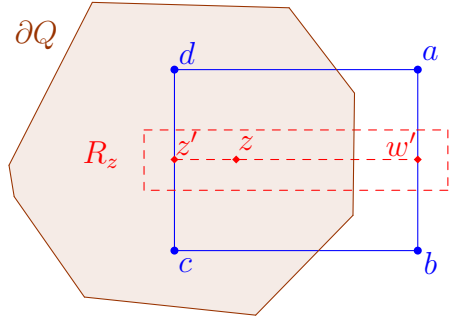
\begin{figure}[htb]
\begin{center}
\input{complex-figure6v2.tex}
\end{center}
\caption{The rectangle $abcd$ and the neighborhood $R_z$.}
\label{fig-strip}
\end{figure}

\begin{lemma}\label{cl-reach-boundary}
A vertex $z\in B^0\cap abcd$ can be joined with $\partial Q$ by a path in the graph $R_z\cap B$.
\end{lemma}

\begin{proof}
Assume that there is no path as required. Then the point $z$ is separated from $\partial Q$ by a path $P$ in the complement $R_z-B$. Since $ab$ is outside $\partial Q$ it follows that $z$ is separated from the point $w'$ by the path $P$ as well. The path $P$ is not closed because the graph $B$ is connected. Thus the endpoints of the path $P$ belong to $\partial R_z$. The path $P$ intersects the segment $zw'$ because $P$ separates $z$ from $w'$. Thus the face of the graph $B$ containing $P$ has  diameter greater than $h$, a contradiction to Diameter Lemma~\ref{cl-diam}.
\end{proof}




\begin{lemma}\label{cl-sym-rectangles} We have $w\in R_z$ if and only if $z\in R_w$.
\end{lemma}

\begin{proof}
Straightforward.
\end{proof}

\begin{lemma}\label{cl-rectangle-estimate}
$h^2 L^{\!2}_{abcd}(u)\le \mathrm{Const}\cdot \left(
e r h L^{\!2}_{\partial Q}(u)+e^4 r^2 E_Q(u)\right)$.
\end{lemma}

\begin{proof} Take a vertex $z\in abcd\cap B^0$.
By Lemma~\ref{cl-reach-boundary} it follows that the vertex $z$ is joined with a vertex $w\in\partial Q$ by a path 
in the graph $R_{z}\cap B$. Take such path with minimal number $m$ of vertices. Then $m$ is not greater than the total number of vertices of the graph $B$ in the rectangle $R_{z}$. Hence $m\le \mathrm{Const}\cdot e r/h$ by Rectangle Capacity Lemma~\ref{cl-points}.
Thus
by Path Energy Lemma~\ref{cl-obvious} we get
$$
u(z)^2\le 2u(w)^2+2(u(z)-u(w))^2\le
2L^{\!2}_{\partial Q\cap R_z}(u)+\mathrm{Const}\cdot e^3rE_{B\cap R_z}(u)/h.
$$
Sum these inequalities over all $z\in abcd\cap B^0$.
A vertex $w\in \partial Q$ can contribute to $L^{\!2}_{\partial Q\cap R_z}(u)$ only if $w\in R_{z}$. Thus by Lemma~\ref{cl-sym-rectangles} and Rectangle Capacity Lemma~\ref{cl-points} it 
contributes 
at most $\mathrm{Const}\cdot er/h$ times.
Similarly, an edge $z''w''$ of the graph $B$ can contribute to $E_{B\cap R_z}(u)$ only if
$z''\in R_{z}$. Thus by Lemma~\ref{cl-sym-rectangles} and Rectangle Capacity Lemma~\ref{cl-points} 
it contributes at most $\mathrm{Const}\cdot er/h$ times.
Thus our summation leads to the required inequality.
\end{proof}






\begin{proof}[Proof of Friedrichs Inequality Lemma~\ref{cl-strip-estimate}] Since $\partial\Omega$ is a smooth curve it follows that its $r$-neighbor\-hood  can be covered by finitely many rectangles $abcd$ such that $ab$ is outside the $r$-neighborhood of $\Omega$ and $|b-c|=3r$. Moreover, the number of rectangles is bounded by a number depending only on the domain $\Omega$ (but not on $K$). These rectangles cover the strip $\Omega-K$ as well. Since $r>\max_{z\in\partial Q}\mathrm{Dist}(z,\partial \Omega)$ and $r>h$ it follows that for each rectangle the side $ab$ is outside $\partial Q$ and $|b-c|>h$. Summing the inequalities of Lemma~\ref{cl-rectangle-estimate} for each of the rectangles, 
we get the required inequality.
\end{proof}

\begin{remark}
Our proof of 
Lemma~\ref{cl-strip-estimate}
does not generalize to domains with nonsmooth boundaries,
e.g., to a domain bounded by a cardioid.
\end{remark}

\begin{corollary} \label{l-boundary-values} 
Let $\Omega$ be a bounded simply-connected domain with smooth boundary. Let $g\colon \mathbb{C}\to\mathbb{R}$ be a smooth function. Let $\{Q_n\}$ be a nondegenerate uniform sequence of finite \emph{quadrilateral} lattices approximating the domain $\Omega$.
Suppose that  the sequence $u_{Q_n,g}\left|_{B_n^0}\right.$ converges to a harmonic function $u\colon \Omega\to \mathbb{R}$ uniformly on each compact subset of $\Omega$. Then $u=u_{\Omega,g}$.
\end{corollary}

\begin{proof}[Proof of Corollary~\ref{l-boundary-values}]
Denote by $S(r)\subset\Omega$ the $r$-neighborhood of  $\partial\Omega$. First let us prove the estimate $h^2 L^{\!2}_{S(r)}(u_{Q_n,g}-g)\le \mathrm{Const}_{\{Q_m\},\Omega,g}r^2$ for $r>h$ and some constant $\mathrm{Const}_{\{Q_m\},\Omega,g}$. 
Apply Friedrichs Inequality Lemma~\ref{cl-strip-estimate} for the function $u_{Q_n,g}-g$. The energy $E_{Q_n}(u_{Q_n,g}-g)\le 2E_{Q_n}(u_{Q_n,g})+2E_{Q_n}(g\left|_{Q^0_n}\right.)$ is bounded by Variational Principle~\ref{cl-laplacian} and Energy Convergence Lemma~\ref{l-energy-convergence}. We have $L^{\!2}_{\partial Q_n}(u_{Q_n,g}-g)=0$ because $u_{Q_n,g}=g$ at~$\partial Q_n$, and the required estimate follows.

Now by Diameter Lemma~\ref{cl-diam} 
it follows that for each $\rho<r$
$$
\int_{S(r)-S(\rho)}(u-g)^2dxdy\preceq
h^2 L^{\!2}_{S(r)-S(\rho)}(u_{Q_n,g}-g)\le
h^2 L^{\!2}_{S(r)}(u_{Q_n,g}-g)\le
\mathrm{Const}_{\{Q_m\},\Omega,g}r^2.
$$
Approaching $\rho\to 0$ we get $\frac{1}{r}\int_{S(r)}(u-g)^2 dxdy\to 0$ as $r\to 0$. By \cite[\S4.1]{Courant-friedrichs-Lewy-28} this condition implies the boundary condition. Thus $u=u_{\Omega,g}$.
\end{proof}

\section{Applications and open problems}\label{sec-open-problems}

\subsection{Application to numerical analysis}\label{ssec-alg}
Convergence Theorem~\ref{th-uniform-approx} provides a new approximation algorithm for the numerical solution of the Dirichlet boundary value problem. It also gives a new convergence result for the classical finite element method, which we are going to state now.

Let $B$ be a triangulation of a polygon $\hat B$. The \emph{finite element method} approximates the solution of the Dirichlet problem on $\Omega$ by a continuous function $u_{B,g}\colon\hat B\to \mathbb{R}$ which is linear on each face of $B$, equal to the given function $g$ on the boundary $B^0\cap \partial\hat B$, and has minimal energy $E_{\hat B}(u)=\int_{\hat B}|\nabla u|^2dxdy$ (among such functions). The restriction $u_{B,g}\colon B^0\to \mathbb{R}$ is 
the \emph{solution of the Dirichlet problem on $B$}.

Equivalently \cite[\S4]{Duffin-59}, 
it can be defined as follows. For a nonboundary edge $z_1z_3\subset B$ denote by $\alpha$ and $\beta$ the angles opposite to the edge $z_1z_3$ in the two triangles of $B$ sharing the edge; see Figure~\ref{fig-cot}. Denote
\vspace{-0.3cm}
\begin{equation}\label{eq-cot}
c(z_1z_3):=(\cot \alpha + \cot \beta)/2.
\end{equation}
Then  $u_{B,g}\colon B^0\to \mathbb{R}$ is the unique function equal to $g$ on $B^0\cap\partial\hat{B}$ such that for each $z_1\in B^0-\partial\hat{B}$ we have $\sum_{z_3}c(z_1z_3)(u_{B,g}(z_1)-u_{B,g}(z_3))=0$,
where the sum is over all neighbors $z_3$ of $z_1$.

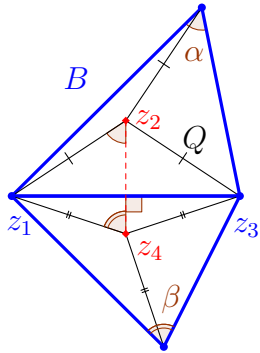
\begin{figure}[htbp]
\begin{center}
\input{complex-figure7v2.tex}
\end{center}
\caption{A ``kite'' lattice $Q$ associated to a Delaunay triangulation $B$.}
\label{fig-cot}
\end{figure}

Usually one proves convergence of the finite element method under certain assumptions on individual triangles \cite{Brandts-etal-11}. For instance, it was proved in \cite[Theorem~3.3.7]{Ciarlet-78} 
that $u_{B_n,g}$ converges uniformly to $u_{\hat B,g}$, if there is a constant $\mathrm{Const}$ such that
\begin{enumerate}
\item[(A)] the minimal angle of each triangle is greater than $1/\mathrm{Const}$;
\item[(R)] the ratio of any two edges of each triangulation is less than $\mathrm{Const}$.
\end{enumerate}
According to~\cite{Brandts-etal-11} no uniform convergence results without assumptions (A) and (R) were available.

Following \cite{Bobenko-Springborn-07} we suggest a new approach measuring ``triangulation quality'' via configuration of neighboring triangles rather than the shape of individual ones. A triangulation $B$ is called \emph{Delaunay}, if $\alpha+\beta\le \pi$ for each pair of adjacent triangular faces (and thus $c(z_1z_3)\ge 0$ above).
A Delaunay triangulation exists for any prescribed set of vertices $B^0$ not contained in one line
~\cite{Bobenko-Springborn-07}. 
A triangulation $B$ has \emph{regular boundary}, if 
each triangular face has no more than one common side with the boundary, and the angle opposite to such common side is less than $\pi/2$. A sequence of triangulations is \emph{nondegenerate uniform}, if there is a constant $\mathrm{Const}$ such that for each member of the sequence
\begin{itemize}
\item[(D)] for each nonboundary edge the sum of opposite angles in the two triangles containing the edge is less than $\pi-1/\mathrm{Const}$ (in particular, the triangulation is Delaunay);
\item[(U)] the number of vertices in an arbitrary disk of radius equal to the maximal edge length is less than $\mathrm{Const}$.
\end{itemize}
Assumption (U) is weaker than (R); neither (D) nor (A) is weaker than the other one.
We prove convergence of the finite element method for triangulations satisfying (D) and (U). 

\begin{corollary}
\label{cor-cotangent}
\textup{}
Let $\Omega\subset\mathbb{C}$ be a bounded simply-connected domain. Let $g\colon \mathbb{C}\to\mathbb{R}$ be a smooth function.
Let $\{B_n\}$ be a nondegenerate uniform sequence of triangulations with regular boundaries approximating the domain $\Omega$. Then the solution $u_{B_n,g}\colon B_n^0\to\mathbb{R}$ of the Dirichlet problem on $B_n$ uniformly converges to the solution $u_{\Omega,g}\colon \mathrm{Cl}\Omega\to\mathbb{R}$
of the Dirichlet problem on $\Omega$.
\end{corollary}

For the proof we need the following lemma.

\begin{lemma}\label{l-transformation}
Let $B$ be a Delaunay triangulation with regular boundary satisfying condition (D) from Section~\ref{ssec-alg}. In each triangle of $B$, draw $3$ segments joining the circumcenter with the vertices. Then the drawn segments do not have common interior points and thus form an 
orthogonal quadrilateral lattice $Q$. For each face $z_1z_2z_3z_4$ of the lattice with the vertices listed clockwise we have 
$
i\frac{z_2-z_4}{z_1-z_3}=(\cot \alpha + \cot \beta)/2,
$
where $\alpha$ and $\beta$ are the angles opposite to the edge $z_1z_3$ in the two triangles of $B$ sharing the edge.
\end{lemma}

\begin{proof} (Cf. \cite[Proof of Lemma~5.4(2)$\Rightarrow$(1)]{Prasolov-Skopenkov-11}.) 
For each oriented 
edge ${z_1z_3}$ of the triangulation $B$ denote by $z_2$ and $z_4$ the circumcenters of the triangular faces bordering upon ${z_1z_3}$ from the left and the right, respectively; see Figure~\ref{fig-cot}. (If ${z_1z_3}$ is a boundary edge then we orient it counterclockwise along the boundary, thus only the point $z_2$ is well-defined).
The required formula is proved by a straightforward computation;
the nontrivial part of the lemma is to prove that the constructed graph $Q$ does not have self-intersections.

For that we are going to compare the sum of the areas of the quadrilaterals $z_1z_2z_3z_4$ (and also the triangles $z_1z_2z_3$ bordering upon boundary edges $z_1z_3$) with the area of their union.
Clearly, 
$\mathrm{Area}(z_1z_2z_3z_4)=|z_1-z_3|^2(\cot \alpha + \cot \beta)/4$ and $\mathrm{Area}(z_1z_2z_3)=|z_1-z_3|^2\cot \alpha/4$
for each interior and boundary edge $z_1z_3$, respectively. 
Since $B$ is Delaunay, has regular boundary, and satisfies condition (D) from Section~\ref{ssec-alg} it follows that these oriented areas are all positive. Expressing the area of each triangular face of $B$ as an algebraic sum of three terms of the form $\mathrm{Area}(z_1z_2z_3)$ we get $\sum \mathrm{Area}(z_1z_2z_3z_4)+\sum \mathrm{Area}(z_1z_2z_3)=\mathrm{Area}(\widehat{B})$, where the sums are over all the interior and boundary edges $z_1z_3$, respectively. Finally, 
$\bigcup z_1z_2z_3z_4\cup\bigcup z_1z_2z_3\supset \widehat{B}$ because $\partial(\bigcup z_1z_2z_3z_4\cup\bigcup z_1z_2z_3)\subset \partial \widehat{B}$. This implies that the quadrilaterals of the form $z_1z_2z_3z_4$ do not overlap, hence $Q$ does not have self-intersections.
\end{proof}

\begin{proof}[Proof of Corollary~\ref{cor-cotangent}] 
Let $Q_n$ be the orthogonal quadrilateral lattice given by Lemma~\ref{l-transformation} for the triangulation $B_n$.
Then $u_{B_n,g}\colon B_n^0\to\mathbb{R}$ is the restriction of a discrete harmonic function $Q_n^0\to \mathbb{R}$.
By Existence and Uniqueness Theorem~\ref{cl-dirichlet-problem} we have $u_{B_n,g}(z)=u_{Q_n,g}(z)$ for each $z\in B_n^0$.
Since $\{B_n\}$ is a nondegenerate uniform sequence 
approximating the domain $\Omega$ it follows that
$\{Q_n\}$ is too, but with condition (D) replaced by relaxed condition (D') from Remark~\ref{rem-relaxed}.
By Convergence Theorem~\ref{th-uniform-approx} and Remark~\ref{rem-relaxed-convergence} the corollary follows.
\end{proof}



Vice versa, the finite element method can be applied to establish convergence of discrete harmonic functions on quadrilateral lattices. Using the standard finite element described above one can approach only rhombic lattices (and also ``kite'' ones at the cost of establishing convergence only at the vertices of the graphs $B_n$ but not $W_n$).

The following \emph{nonconforming} finite element might be useful in the case of general quadrilateral lattices.
Given a function $u\colon Q^0\to\mathbb{R}$ define
its \emph{interpolation} $I_Q u\colon z_1z_2z_3z_4\to\mathbb{R}$
to be the linear function
on a face $z_1z_2z_3z_4$ of~$Q$
such that $[I_Q u](z_1)=u(z_1)$, $[I_Q u](z_3)=u(z_3)$, and $[I_Q u](z_2)-[I_Q u](z_4)=u(z_2)-u(z_4)$.
Combining such linear functions together we get a (discontinuous) function $I_Q u\colon \widehat{Q}\to\mathbb{R}$ on the union $\widehat{Q}$ of all the quadrilateral faces of~$Q$. 
%
Clearly, then
$E_Q(u)=E_{\widehat{Q}}(I_Q u)$.
An interesting observation is that
in the case when all the bounded faces of $Q$ are convex 
we have $E_Q(u)=\int_{\widehat{Q}}\nabla I_{B\cup Q} u\cdot \nabla I_{W\cup Q} u\, dxdy$, 
where $I_{B\cup Q} u,I_{W\cup Q} u\colon \widehat{Q}\to\mathbb{R}$
are piecewise linear extensions of the function $u\colon Q^0\to \mathbb{R}$ to the faces of the graphs $B\cup Q$ and $W\cup Q$, respectively.

\begin{problem} Give an effective approximation algorithm for finding the solution of the Dirichlet problem on a (nonorthogonal) quadrilateral lattice.
\end{problem}


\begin{problem} Estimate the rate of convergence in Convergence Theorem~\ref{th-uniform-approx}.
\end{problem}

\subsection{Physical interpretation}\label{ssec-phys}


Classical physical interpretation of complex analysis on orthogonal lattices uses direct-current networks \cite{Duffin-68} (for elementary introduction to networks see \cite{Dorichenko-Prasolov-Skopenkov, Skopenkov-Smykalov-Ustinov, Prasolov-Skopenkov-11}). Let us give a new physical interpretation for arbitrary quadrilateral lattices involving \emph{alternating}-current networks. The interpretation gives some motivation for our definitions and it is also interesting in itself.

Define the \emph{admittance} of an edge $z_1z_3\subset B$ by the formula
\begin{equation}\label{eq-def-conductance}
c(z_1z_3):=i\frac{z_2-z_4}{z_1-z_3},
\end{equation}
where $z_1z_2z_3z_4$ is the face containing $z_1z_3$ with the vertices listed clockwise. Clearly, this number has positive real part (and in case of an orthogonal lattice it is simply a positive number).

A graph $B$ with edge admittances having positive real parts can be considered as an \emph{alternating-current network}; see 
\cite[\S2.4]{Prasolov-Skopenkov-11} and \cite{Foster}.

Given a discrete analytic function $f\colon Q^0\to\mathbb{C}$,
define the \emph{voltage drop} $V(z_1z_3)$ and the \emph{current} $I(z_1z_3)$ on an oriented edge $z_1z_3\subset B$ by the formula
\begin{equation*}
V(z_1z_3):=f(z_1)-f(z_3),
\qquad I(z_1z_3):=if(z_2)-if(z_4),
\end{equation*}
where $z_1z_2z_3z_4$ is the face of the lattice $Q$ with the vertices listed clockwise.
\emph{Boundary voltage drops} are the differences $f(z_1)-f(z_3)$ for all pairs of consecutive boundary vertices $z_1, z_3\in B^0\cap \partial Q$. \emph{Boundary currents} (or \emph{incoming currents}) are the values $if(z_4)-if(z_2)$ for all the pairs of consecutive boundary vertices $z_2, z_4\in W^0\cap\partial Q$.
The \emph{voltage drop at a moment $t$} is the number $\Real \left(V(z_1z_3)\exp(it)\right)$; the \emph{current at the moment $t$} is defined analogously.

A reformulation of Existence and Uniqueness Theorem~\ref{cl-dirichlet-problem} is the following result.

\begin{corollary}\label{cor-physical}
Boundary voltage drops at the initial moment and boundary currents after one quarter of the period
uniquely determine all the voltage drops and currents in an alternating-current network at all the moments of time.
\end{corollary}




The physical meaning of Convergence Theorem~\ref{th-uniform-approx} 
is that the voltage in a distributed direct-current network can be approximated by voltages in lumped direct-current 
networks.

This physical interpretation gives also one more motivation for the definition of energy from Section~\ref{ssec-energy}. The \emph{energy} of the network (\emph{dissipated per period}) is
\begin{equation}\label{eq-def-alternating}
E_Q(\mathrm{Re}f):=\Real \sum_{z_1z_3}\ V(z_1z_3)\bar I(z_1z_3)/2,
\end{equation}
where the sum is over all the edges $z_1z_3\subset B$.
One can see that the energy can indeed be expressed 
through $\Real f$ only, and formula~\eqref{eq-def-alternating} is equivalent to formula~\eqref{eq-def-energy}  in Section~\ref{ssec-energy} in the particular case of a discrete harmonic function $u=\Real f$.
A restatement of \cite[Claim~5.1]{Prasolov-Skopenkov-11} is  the following result.

\begin{energy-conservation-principle} \label{cl-energy-conservation}
For any discrete analytic function $f\colon Q^0\to \mathbb{C}$ we have
$$
E_Q(\mathrm{Re}f)=
\mathrm{Im}\sum_{z_3\in\partial Q\cap B^0}f(z_3)(\bar f(z_2)-\bar f(z_4))/2,
$$
where the sum is over all the boundary vertices $z_3\in\partial Q\cup B^0$ with boundary neighbors $z_2$ and $z_4$ in counterclockwise order.
\end{energy-conservation-principle}

\subsection{Probabilistic interpretation}\label{ssec-prob}

Probabilistic interpretation of discrete harmonic functions has been discussed already in \cite{Courant-friedrichs-Lewy-28}; see \cite{Skopenkov-Smykalov-Ustinov} for an elementary introduction.
The results of the present paper allow to generalize many estimates from \cite{Chelkak-Smirnov-08} to nonrhombic lattices. As an example let us prove convergence of the discrete harmonic measure to its continuous counterpart and sketch one particular problem.

Let $Q$ be an orthogonal lattice and let $B$ be one of the connected graphs obtained by joining the opposite vertices in each quadrilateral face of~$Q$. A \emph{random walk} on the vertices of~$B$ is defined as follows. At each moment of time the walker moves from his current position to one of the neighboring vertices with the probability proportional to the weights of the corresponding edges given by formula~(\ref{eq-def-conductance}). 

Let $E\subset\partial Q$ be an arc such that $\partial E\cap Q^0=\emptyset$.
The probability that a random walk starting at a vertex $z\in B^0$ (or $z\in W^0$) first hits the boundary $\partial Q$ at some vertex of the arc $E$ is called the \emph{discrete harmonic measure} $\omega(z,E,Q)$ of the arc $E$. Equivalently, $\omega(z,E,Q)=u_{Q,\chi_{E}}(z)$ is the solution of the Dirichlet boundary value problem on $Q$ for the characteristic function $\chi_{E}\colon\mathbb{C}\to \{0;1\}$ of the arc $E$. 
Its continuous counterpart is the \emph{harmonic measure} $\omega(z,E,\Omega)$ of an arc $E\subset\partial \Omega$, i.e., 
a continuous function $\mathrm{Cl}\Omega-\partial E\to\mathbb{R}$ harmonic inside $\Omega$ and equal to $\chi_{E}$ at the boundary $\partial\Omega-\partial E$. 

We say that a sequence of arcs $E_n$ \emph{approximates} an arc $E$, if the maximal distance from a point of $E_n$ to the arc $E$ and the maximal distance from a point of $E$ to the arc $E_n$ tend to zero as $n\to\infty$. A sequence of functions $u_{n}\colon Q^0_n\to\mathbb{R}$ \emph{converges} to a function $u\colon\Omega\to\mathbb{R}$ \emph{uniformly on each compact subset}, if for each compact set $K\subset \Omega$ we have $\max_{z\in K\cap Q_n^0}|u_n(z)-u(z)|\to 0$ as $n\to\infty$.

\begin{corollary}\label{cor-measure}
Let $\Omega\subset\mathbb{C}$ be a domain bounded by a continuous closed curve $\partial\Omega$ without self-intersections. 
Let $\{Q_n\}$ be a nondegenerate uniform sequence of 
orthogonal lattices  approximating the domain $\Omega$. 
Let $E_n\subset \partial Q_n$ be a sequence of arcs approximating an arc $E\subset \partial \Omega$ 
such that $\partial Q_n-E_n$ approximate $\partial \Omega-E$ as well.
Then the discrete harmonic measure $\omega(\cdot,E_n,Q_n)\colon Q_n^0\to\mathbb{R}$  converges to the harmonic measure $\omega(\cdot,E,\Omega)\colon \Omega\to\mathbb{R}$ uniformly on each compact subset of $\Omega$.
\end{corollary}

\begin{proof}[Proof of Corollary~\ref{cor-measure}]
For a set $X\subset \mathbb{C}$ denote by $N_\epsilon(X)$ its $\epsilon$-neighborhood in $\mathbb{C}$. 
For each 
$\epsilon>0$ take a smooth function 
$g_\epsilon\colon\mathbb{C}\to [0,1]$ equal to $1$ at the set $E-N_{2\epsilon}(\partial\Omega-E)$ and vanishing outside
the $\epsilon$-neighborhood of the set.

Take a compact set $K\subset\Omega$ and a number $\delta>0$.
It suffices to prove the following sequence of estimates for some $\epsilon$,
$n_0$ (depending on $K,\delta,\Omega,E,\{Q_m\},\{E_m\}$) and each $n>n_0$, $z\in K\cap Q^0_n$:
$$
\omega(z,E_n,Q_n)
\ge u_{Q_n,g_\epsilon}(z)
\ge u_{\Omega,g_\epsilon}(z)-\delta
\ge \omega(z,E,\Omega)-2\delta.
$$
Together with analogous estimate for
$\omega(z,\partial Q_n-E_n,Q_n)=1-\omega(z,E_n,Q_n)$
this implies the corollary.

Here the first estimate holds for large enough $n$ such that the Hausdorff distance between $\partial Q_n-E_n$ and $\partial \Omega-E$ is less than~$\epsilon$ by the inequality $\chi_{E_n}\left|_{\partial Q_n^0}\right.\ge g_\epsilon$ and Maximum Principle~\ref{cl-maximum-principle}. The second estimate holds for $n$ starting from some $n_0$ (depending on $\epsilon$ and $\delta$) by Convergence Theorem~\ref{th-uniform-approx}. By the Lindel\" of maximum principle, the Dini theorem, and  $\sigma$-additivity of the harmonic measure we have
$$
\omega(z,E,\Omega)-u_{\Omega,g_\epsilon}(z)\le \omega(z, N_{2\epsilon}(\partial\Omega - E)\cap E,\Omega)\stackrel{K}{\rightrightarrows} \omega(z, \bigcap_{m}N_{2^{-m}}(\partial\Omega - E)\cap E,\Omega)=0 
\quad\text{as} \quad\epsilon\to 0.
$$
This implies the last estimate above for some $\epsilon$ (depending on
$K$, $\delta$, $\Omega$, $E$).
\end{proof}

\begin{problem} 
Prove that the trajectories of loop-erased random walks on the orthogonal lattice converge to $\mathrm{SLE}_2$ curves in the scaling limit (see \cite{Chelkak-Smirnov-08, Lawler-Schramm-Werner-04} for the definitions).
\end{problem}

\subsection{Generalizations}



\begin{problem} \label{prob-general} Generalize Theorem~\ref{th-uniform-approx} to:
\begin{enumerate}[(1)]
    \item nonorthogonal quadrilateral lattices;
    \item nonuniform sequences, i.e., not satisfying condition (U) from Section~\ref{ssec-state} (for \emph{adaptive meshes});
    \item \label{prob-general-discontinuous} singular boundary values (for convergence of 
    Green's function, the Cauchy and the Poisson kernels, Abelian integrals);
    \item other types of boundary conditions;
    \item other Riemann surfaces; 
    \item higher dimensions;
    \item other elliptic PDE.
\end{enumerate}
\end{problem}

\begin{problem} Prove that under the assumptions of Convergence Theorem~\ref{th-uniform-approx} the gradient $\nabla_{\!Q_n} u_{Q_n,g}$ converges to $\nabla u_{\Omega,g}$ uniformly on each compact subset of $\Omega$.
\end{problem}













\begin{problem} Construct a sequence of quadrilateral lattices approximating a planar domain such that the solutions of the Dirichlet problem on the lattices do \emph{not} converge uniformly to the solution of the Dirichlet problem in the domain.
\end{problem}

\begin{problem} For which nonorthogonal lattices $Q$ Maximum Principle~\ref{cl-maximum-principle} remains true?
\end{problem}

\subsection*{Acknowledgements}

The author is grateful to A.~Bobenko, D.~Chelkak, C.~Mercat, S. Novikov, A.~Pakharev, S.~Tikhomirov, and A.~Ustinov for useful discussions. This work has been presented at the seminars of A. Bobenko, S. Novikov, Ya. Sinai, and S. Smirnov.



\bibliographystyle{elsarticle-num}

\end{document}

%% file: dirichlet-figure1a.tex
\definecolor{uququq}{rgb}{0.25,0.25,0.25}
\begin{tikzpicture}[line cap=round,line join=round,>=triangle 45,x=0.17cm,y=0.17cm]
\clip(-0.61,-6.8) rectangle (12.8,7.24);
\draw (4,6)-- (4,-6);
\draw (8,6)-- (8,-6);
\draw (0,2)-- (12,2);
\draw (0,-2)-- (12,-2);
\draw (0,2)-- (0,-2);
\draw (12,2)-- (12,-2);
\draw (4,6)-- (8,6);
\draw (4,-6)-- (8,-6);
\draw (0,6)-- (4,6);
\draw (0,6)-- (0,2);
\draw (8,6)-- (12,6);
\draw (12,6)-- (12,2);
\draw (12,-2)-- (12,-6);
\draw (12,-6)-- (8,-6);
\draw (4,-6)-- (0,-6);
\draw (0,-2)-- (0,-6);
\fill [color=uququq] (4,6) circle (1.5pt);
\fill [color=uququq] (4,-6) circle (1.5pt);
\draw[color=black] (5.78,0.16) node {$Q$};
\fill [color=uququq] (8,6) circle (1.5pt);
\fill [color=uququq] (8,-6) circle (1.5pt);
\fill [color=uququq] (0,2) circle (1.5pt);
\fill [color=uququq] (12,2) circle (1.5pt);
\fill [color=uququq] (0,-2) circle (1.5pt);
\fill [color=uququq] (12,-2) circle (1.5pt);
\fill [color=uququq] (4,2) circle (1.5pt);
\fill [color=uququq] (8,2) circle (1.5pt);
\fill [color=uququq] (8,-2) circle (1.5pt);
\fill [color=uququq] (4,-2) circle (1.5pt);
\fill [color=uququq] (0,6) circle (1.5pt);
\fill [color=uququq] (12,6) circle (1.5pt);
\fill [color=uququq] (12,-6) circle (1.5pt);
\fill [color=uququq] (0,-6) circle (1.5pt);
\end{tikzpicture} 

%% file: dirichlet-figure1b.tex
\definecolor{uququq}{rgb}{0.25,0.25,0.25}
\begin{tikzpicture}[line cap=round,line join=round,>=triangle 45,x=0.17cm,y=0.17cm]
\clip(-0.61,-7.4) rectangle (13.25,7.24);
\draw (5.99,3.16)-- (6,0);
\draw (3.27,-1.59)-- (6,0);
\draw (3.27,-1.59)-- (6.01,-3.16);
\draw (8.74,-1.57)-- (6.01,-3.16);
\draw (6,0)-- (8.74,-1.57);
\draw (8.73,1.59)-- (8.74,-1.57);
\draw (5.99,3.16)-- (8.73,1.59);
\draw (3.26,1.57)-- (5.99,3.16);
\draw (3.27,-1.59)-- (3.26,1.57);
\draw (3.27,-1.59)-- (3.28,-4.75);
\draw (0.53,-3.18)-- (3.28,-4.75);
\draw (0.52,-0.02)-- (0.53,-3.18);
\draw (0.52,-0.02)-- (3.27,-1.59);
\draw (0.51,3.15)-- (0.52,-0.02);
\draw (3.26,1.57)-- (0.51,3.15);
\draw (3.25,4.74)-- (0.51,3.15);
\draw (3.25,4.74)-- (5.99,3.16);
\draw (3.25,4.74)-- (5.98,6.33);
\draw (5.99,3.16)-- (8.72,4.75);
\draw (5.98,6.33)-- (8.72,4.75);
\draw (11.47,3.18)-- (8.72,4.75);
\draw (11.47,3.18)-- (8.73,1.59);
\draw (11.47,3.18)-- (11.48,0.02);
\draw (8.74,-1.57)-- (11.48,0.02);
\draw (8.74,-1.57)-- (8.75,-4.73);
\draw (11.48,0.02)-- (11.49,-3.14);
\draw (8.75,-4.73)-- (11.49,-3.14);
\draw (6.01,-3.16)-- (6.02,-6.32);
\draw (8.75,-4.73)-- (6.02,-6.32);
\draw (3.28,-4.75)-- (6.02,-6.32);
\fill [color=black] (6,0) ++(-1.5pt,0 pt) -- ++(1.5pt,1.5pt)--++(1.5pt,-1.5pt)--++(-1.5pt,-1.5pt)--++(-1.5pt,1.5pt);
\draw[color=black] (4.73,1.09) node {$Q$};
\fill [color=black] (6.02,-6.32) circle (1.5pt);
\fill [color=black] (11.49,-3.14) circle (1.5pt);
\fill [color=uququq] (11.47,3.18) circle (1.5pt);
\fill [color=uququq] (5.98,6.33) circle (1.5pt);
\fill [color=uququq] (0.51,3.15) circle (1.5pt);
\fill [color=uququq] (0.53,-3.18) circle (1.5pt);
\fill [color=uququq] (5.99,3.16) circle (1.5pt);
\fill [color=uququq] (3.27,-1.59) circle (1.5pt);
\fill [color=uququq] (8.74,-1.57) circle (1.5pt);
\fill [color=uququq] (6.01,-3.16) ++(-1.5pt,0 pt) -- ++(1.5pt,1.5pt)--++(1.5pt,-1.5pt)--++(-1.5pt,-1.5pt)--++(-1.5pt,1.5pt);
\fill [color=uququq] (5.99,3.16) circle (1.5pt);
\fill [color=uququq] (0.52,-0.02) circle (1.5pt);
\fill [color=uququq] (3.25,4.74) circle (1.5pt);
\fill [color=uququq] (8.72,4.75) circle (1.5pt);
\fill [color=uququq] (11.48,0.02) circle (1.5pt);
\fill [color=uququq] (8.75,-4.73) circle (1.5pt);
\fill [color=uququq] (3.28,-4.75) circle (1.5pt);
\fill [color=uququq] (3.26,1.57) circle (1.5pt);
\fill [color=uququq] (8.73,1.59) circle (1.5pt);
\end{tikzpicture} 

%% file: dirichlet-figure1c.tex
\definecolor{tttttt}{rgb}{0.2,0.2,0.2}
\definecolor{uququq}{rgb}{0.25,0.25,0.25}
\begin{tikzpicture}[line cap=round,line join=round,>=triangle 45,x=0.19cm,y=0.19cm]
\clip(-0.61,-6.59) rectangle (12.92,7.24);
\draw(8,-0.64) -- (8.64,-0.64) -- (8.64,0) -- (8,0) -- cycle; 
\draw(2.95,3.26) -- (2.65,3.82) -- (2.09,3.51) -- (2.39,2.96) -- cycle; 
\draw(2.44,-1.67) -- (2.9,-2.11) -- (3.34,-1.64) -- (2.87,-1.21) -- cycle; 
\draw(10.94,-1.34) -- (11.48,-1.68) -- (11.82,-1.14) -- (11.28,-0.81) -- cycle; 
\draw(11.34,1.92) -- (10.85,1.52) -- (11.25,1.03) -- (11.74,1.43) -- cycle; 
\draw [dash pattern=on 1pt off 1pt,color=tttttt] (12.28,0.78)-- (8,-6);
\draw [dash pattern=on 1pt off 1pt,color=tttttt] (8,6)-- (12.28,0.78);
\draw [dash pattern=on 1pt off 1pt,color=tttttt] (-0.12,1.59)-- (8,-6);
\draw [dash pattern=on 1pt off 1pt,color=tttttt] (-0.12,1.59)-- (8,6);
\draw [dash pattern=on 1pt off 1pt,color=tttttt] (4,0)-- (0.73,-3.5);
\draw [dash pattern=on 1pt off 1pt,color=tttttt] (4,0)-- (10,0);
\draw [dash pattern=on 1pt off 1pt,color=tttttt] (8,6)-- (8,-6);
\draw [dash pattern=on 1pt off 1pt,color=tttttt] (1.56,4.5)-- (4,0);
\draw [dash pattern=on 1pt off 1pt,color=tttttt] (12.09,1.71)-- (10,0);
\draw [dash pattern=on 1pt off 1pt,color=tttttt] (10,0)-- (12.17,-1.37);
\draw (1.56,4.5)-- (8,6);
\draw (4,0)-- (8,6);
\draw (-0.12,1.59)-- (4,0);
\draw (1.56,4.5)-- (-0.12,1.59);
\draw (-0.12,1.59)-- (0.73,-3.5);
\draw (4,0)-- (8,-6);
\draw (0.73,-3.5)-- (8,-6);
\draw (8,6)-- (10,0);
\draw (10,0)-- (8,-6);
\draw (8,6)-- (12.09,1.71);
\draw (12.09,1.71)-- (12.28,0.78);
\draw (10,0)-- (12.28,0.78);
\draw (12.28,0.78)-- (12.17,-1.37);
\draw (12.17,-1.37)-- (8,-6);
\fill [color=uququq] (8,-6) circle (1.5pt);
\fill [color=uququq] (12.28,0.78) circle (1.5pt);
\fill [color=uququq] (10,0) circle (1.5pt);
\fill [color=uququq] (8,6) circle (1.5pt);
\fill [color=uququq] (-0.12,1.59) circle (1.5pt);
\fill [color=uququq] (4,0) circle (1.5pt);
\fill [color=uququq] (0.73,-3.5) circle (1.5pt);
\draw[color=tttttt] (6.8,1.27) node {$Q$};
\fill [color=uququq] (12.09,1.71) circle (1.5pt);
\fill [color=uququq] (12.17,-1.37) circle (1.5pt);
\fill [color=uququq] (1.56,4.5) circle (1.5pt);
\end{tikzpicture}

%% file: complex-figure3d.tex
\begin{tikzpicture}[line cap=round,line join=round,>=triangle 45,x=0.19cm,y=0.19cm]
\clip(0.2,-7.91) rectangle (12.59,6.28);
\draw (5.66,-2.36)-- (7.52,5.01);
\draw (7.52,5.01)-- (2.86,3.92);
\draw (2.86,3.92)-- (0.73,0.73);
\draw (0.73,0.73)-- (5.66,-2.36);
\draw (0.73,0.73)-- (3.16,-5.62);
\draw (5.66,-2.36)-- (7.57,-6.48);
\draw (3.16,-5.62)-- (7.57,-6.48);
\draw (9.64,-0.56)-- (7.57,-6.48);
\draw (10.76,-4.65)-- (7.57,-6.48);
\draw (12.21,-0.02)-- (10.76,-4.65);
\draw (12.21,-0.02)-- (9.64,-0.56);
\draw (12.21,-0.02)-- (11.23,2.57);
\draw (7.52,5.01)-- (9.64,-0.56);
\draw (7.52,5.01)-- (11.23,2.57);
\fill [color=black] (2.86,3.92) circle (2.0pt);
\fill [color=black] (5.66,-2.36) circle (2.0pt);
\draw[color=black] (4.43,-2.66) node {$z_1$};
\fill [color=black] (3.16,-5.62) circle (2.0pt);
\fill [color=black] (9.64,-0.56) circle (2.0pt);
\draw[color=black] (10.76,-1.52) node {$z_3$};
\fill [color=black] (11.23,2.57) circle (2.0pt);
\fill [color=black] (10.76,-4.65) circle (2.0pt);
\fill [color=black] (0.73,0.73) circle (2.0pt);
\fill [color=black] (7.52,5.01) circle (2.0pt);
\draw[color=black] (8.84,5.56) node {$z_2$};
\fill [color=black] (12.21,-0.02) circle (2.0pt);
\fill [color=black] (7.57,-6.48) circle (2.0pt);
\draw[color=black] (8.75,-7.07) node {$z_4$};
\draw[color=black] (5.3,1.99) node {$Q$};
\end{tikzpicture}

%% file: complex-figure3v2.tex
\begin{tikzpicture}[line cap=round,line join=round,>=triangle 45,x=0.2cm,y=0.2cm]
\clip(0.2,-7.55) rectangle (12.59,6.16);
\draw (5.66,-2.36)-- (7.52,5.01);
\draw (7.52,5.01)-- (2.86,3.92);
\draw (2.86,3.92)-- (0.73,0.73);
\draw (0.73,0.73)-- (5.66,-2.36);
\draw (0.73,0.73)-- (3.16,-5.62);
\draw (5.66,-2.36)-- (7.57,-6.48);
\draw (3.16,-5.62)-- (7.57,-6.48);
\draw (9.64,-0.56)-- (7.57,-6.48);
\draw (10.76,-4.65)-- (7.57,-6.48);
\draw (12.21,-0.02)-- (10.76,-4.65);
\draw (12.21,-0.02)-- (9.64,-0.56);
\draw (12.21,-0.02)-- (11.23,2.57);
\draw (7.52,5.01)-- (9.64,-0.56);
\draw (7.52,5.01)-- (11.23,2.57);
\fill [color=black] (2.86,3.92) circle (2.0pt);
\fill [color=black] (5.66,-2.36) circle (2.0pt);
\draw[color=black] (4.85,-2.72) node {$z_1$};
\fill [color=black] (3.16,-5.62) circle (2.0pt);
\fill [color=black] (9.64,-0.56) circle (2.0pt);
\draw[color=black] (10.55,-1.34) node {$z_3$};
\fill [color=black] (11.23,2.57) circle (2.0pt);
\fill [color=black] (10.76,-4.65) circle (2.0pt);
\fill [color=black] (0.73,0.73) circle (2.0pt);
\fill [color=black] (7.52,5.01) circle (2.0pt);
\draw[color=black] (8.84,5.56) node {$z_2$};
\fill [color=black] (12.21,-0.02) circle (2.0pt);
\fill [color=black] (7.57,-6.48) circle (2.0pt);
\draw[color=black] (8.75,-6.86) node {$z_4$};
\draw[color=black] (5.3,1.99) node {$Q$};
\end{tikzpicture}

%% file: complex-figure1v2.tex
\definecolor{ffqqqq}{rgb}{1,0,0}
\definecolor{qqqqff}{rgb}{0,0,1}
\begin{tikzpicture}[line cap=round,line join=round,>=triangle 45,x=0.2cm,y=0.2cm]
\clip(0.26,-7.43) rectangle (12.8,6.13);
\draw [color=qqqqff] (2.86,3.92)-- (5.66,-2.36);
\draw [color=qqqqff] (5.66,-2.36)-- (3.16,-5.62);
\draw [color=qqqqff] (5.66,-2.36)-- (9.64,-0.56);
\draw [color=qqqqff] (9.64,-0.56)-- (11.23,2.57);
\draw [color=qqqqff] (9.64,-0.56)-- (10.76,-4.65);
\draw [dash pattern=on 1pt off 1pt,color=ffqqqq] (0.73,0.73)-- (7.52,5.01);
\draw [dash pattern=on 1pt off 1pt,color=ffqqqq] (7.52,5.01)-- (12.21,-0.02);
\draw [dash pattern=on 1pt off 1pt,color=ffqqqq] (12.21,-0.02)-- (7.57,-6.48);
\draw [dash pattern=on 1pt off 1pt,color=ffqqqq] (0.73,0.73)-- (7.57,-6.48);
\draw [dash pattern=on 1pt off 1pt,color=ffqqqq] (7.52,5.01)-- (7.57,-6.48);
\fill [color=qqqqff] (2.86,3.92) circle (2.0pt);
\fill [color=qqqqff] (5.66,-2.36) circle (2.0pt);
\draw[color=qqqqff] (4.34,-1.82) node {$z_1$};
\draw[color=qqqqff] (3.14,0.46) node {$B$};
\fill [color=qqqqff] (3.16,-5.62) circle (2.0pt);
\fill [color=qqqqff] (9.64,-0.56) circle (2.0pt);
\draw[color=qqqqff] (10.85,-0.35) node {$z_3$};
\fill [color=qqqqff] (11.23,2.57) circle (2.0pt);
\fill [color=qqqqff] (10.76,-4.65) circle (2.0pt);
\fill [color=ffqqqq] (0.73,0.73) ++(-2.0pt,0 pt) -- ++(2.0pt,2.0pt)--++(2.0pt,-2.0pt)--++(-2.0pt,-2.0pt)--++(-2.0pt,2.0pt);
\fill [color=ffqqqq] (7.52,5.01) ++(-2.0pt,0 pt) -- ++(2.0pt,2.0pt)--++(2.0pt,-2.0pt)--++(-2.0pt,-2.0pt)--++(-2.0pt,2.0pt);
\draw[color=ffqqqq] (8.84,5.56) node {$z_2$};
\fill [color=ffqqqq] (12.21,-0.02) ++(-2.0pt,0 pt) -- ++(2.0pt,2.0pt)--++(2.0pt,-2.0pt)--++(-2.0pt,-2.0pt)--++(-2.0pt,2.0pt);
\fill [color=ffqqqq] (7.57,-6.48) ++(-2.0pt,0 pt) -- ++(2.0pt,2.0pt)--++(2.0pt,-2.0pt)--++(-2.0pt,-2.0pt)--++(-2.0pt,2.0pt);
\draw[color=ffqqqq] (8.75,-6.86) node {$z_4$};
\draw[color=ffqqqq] (8.57,1.63) node {$W$};
\end{tikzpicture}

%% file: complex-figure8v3.tex
\definecolor{uququq}{rgb}{0.25,0.25,0.25}
\begin{tikzpicture}[line cap=round,line join=round,>=triangle 45,x=0.4cm,y=0.4cm]
\clip(-3.94,-1.26) rectangle (11.2,3.24);
\draw (0,3)-- (2,1);
\draw (0,3)-- (4,2);
\draw (4,2)-- (4,1);
\draw (2,1)-- (4,1);
\draw (2,1)-- (0,-1);
\draw (4,0)-- (0,-1);
\draw (4,1)-- (4,0);
\draw (4,2)-- (8,3);
\draw (6,1)-- (8,3);
\draw (4,1)-- (6,1);
\draw (4,0)-- (8,-1);
\draw (6,1)-- (8,-1);
\draw (0.26,0.56) node {$M$}; 
\draw (0.72,0.52)-- (3.78,0.88);
\fill [color=uququq] (0,3) circle (1.5pt);
\draw[color=uququq] (-0.6,2.7) node {$0$}; 
\fill [color=uququq] (8,-1) circle (1.5pt);
\draw[color=uququq] (8.5,-0.54) node {$0$}; 
\fill [color=uququq] (8,3) circle (1.5pt);
\draw[color=uququq] (8.5,2.7) node {$0$};
\fill [color=uququq] (0,-1) circle (1.5pt);
\draw[color=uququq] (-0.6,-0.54) node {$0$};
\fill [color=uququq] (4,2) circle (1.5pt);
\draw[color=uququq] (4.28,2.5) node {$1$};
\fill [color=uququq] (2,1) circle (1.5pt);
\draw[color=uququq] (2.26,1.44) node {$0$};
\fill [color=uququq] (6,1) circle (1.5pt);
\draw[color=uququq] (5.88,0.44) node {$0$};
\fill [color=uququq] (4,0) circle (1.5pt);
\draw[color=uququq] (4.26,-0.54) node {$1$};
\fill [color=uququq] (4,1) circle (1.5pt);
\end{tikzpicture}

%% file: complex-figure4v4.tex
\definecolor{ffffff}{rgb}{1,1,1}
\definecolor{zzttqq}{rgb}{0.6,0.2,0}
\definecolor{qqqqff}{rgb}{0,0,1}
\begin{tikzpicture}[line cap=round,line join=round,>=triangle 45,x=0.8cm,y=0.8cm]
\draw [color=qqqqff,dash pattern=on 1pt off 1pt, xstep=0.8cm,ystep=0.8cm] (-2.81,-0.77) grid (6.81,6.73);
\clip(-2.81,-0.77) rectangle (6.81,6.73);
\draw [line width=1.6pt,color=zzttqq] (1,3)-- (4,3);
\draw [line width=1.6pt,color=zzttqq] (0,4)-- (5,4);
\draw [line width=1.6pt,color=zzttqq] (5,4)-- (5,2);
\draw [line width=1.6pt,color=zzttqq] (5,2)-- (0,2);
\draw [line width=1.6pt,color=zzttqq] (0,2)-- (0,4);
\draw [line width=1.6pt,color=zzttqq] (-1,5)-- (6,5);
\draw [line width=1.6pt,color=zzttqq] (6,1)-- (-1,1);
\draw [line width=1.6pt,color=zzttqq] (-1,1)-- (-1,5);
\draw [line width=1.6pt,color=zzttqq] (-2,6)-- (-2,0);
\draw [line width=1.6pt,color=zzttqq] (-2,0)-- (6,0);
\draw [line width=1.6pt,color=zzttqq] (-2,6)-- (6,6);
\draw (5,6)-- (5,5);
\draw [line width=1.6pt,dash pattern=on 1pt off 1pt on 2pt off 4pt,color=qqqqff] (6,-0.77) -- (6,6.73);
\draw [color=ffffff] (6,6)-- (7,6);
\draw [color=ffffff] (6,5)-- (7,5);
\draw [color=ffffff] (6,4)-- (7,4);
\draw [color=ffffff] (6,3)-- (7,3);
\draw [color=ffffff] (6,2)-- (7,2);
\draw [color=ffffff] (6,1)-- (7,1);
\draw [color=ffffff] (6,0)-- (7,0);
\fill [color=qqqqff] (1,3) circle (1.5pt);
\draw[color=qqqqff] (1.2,3.31) node {$z$};
\fill [color=qqqqff] (4,3) circle (1.5pt);
\draw[color=qqqqff] (4.33,3.32) node {$w$};
\fill [color=qqqqff] (0,2) circle (1.5pt);
\draw[color=qqqqff] (0.39,1.56) node {$z_1$};
\draw[color=zzttqq] (2.51,1.64) node {$R_1$};
\draw[color=zzttqq] (2.52,0.63) node {$R_2$};
\fill [color=qqqqff] (6,0) circle (1.5pt);
\fill [color=qqqqff] (0,0) circle (1.5pt);
\draw[color=qqqqff] (0.37,-0.37) node {$z_3$};
\draw[color=black] (5.35,5.55) node {$h$};
\draw[color=qqqqff] (6.45,5.65) node {$\partial B$};
\fill [color=qqqqff] (6,0) circle (1.5pt);
\draw[color=qqqqff] (5.49,-0.4) node {$w_3$};
\fill [color=qqqqff] (4,4) circle (1.5pt);
\draw[color=qqqqff] (4.45,4.32) node {$w_1$};
\end{tikzpicture}

%% file: complex-figure6v2.tex
\definecolor{zzttqq}{rgb}{0.6,0.2,0}
\definecolor{ffqqqq}{rgb}{1,0,0}
\definecolor{qqqqff}{rgb}{0,0,1}
\begin{tikzpicture}[line cap=round,line join=round,>=triangle 45,x=0.4cm,y=0.4cm]
\clip(0.23,-6.8) rectangle (15.26,5.5);
\fill[color=zzttqq,fill=zzttqq,fill opacity=0.1] (11.92,1.22) -- (11.87,-2.83) -- (8.67,-6.14) -- (3.04,-5.54) -- (0.73,-2.21) -- (0.56,-1.17) -- (3.3,4.23) -- (9.77,4.05) -- cycle;
\draw [color=qqqqff] (14,2)-- (14,-4);
\draw [color=qqqqff] (14,-4)-- (6,-4);
\draw [color=qqqqff] (6,-4)-- (6,2);
\draw [color=qqqqff] (6,2)-- (14,2);
\draw [dash pattern=on 3pt off 3pt,color=ffqqqq] (6,-0.98)-- (14,-0.98);
\draw [dash pattern=on 3pt off 3pt,color=ffqqqq] (5,0)-- (14.99,0);
\draw [dash pattern=on 3pt off 3pt,color=ffqqqq] (14.99,0)-- (14.99,-2);
\draw [dash pattern=on 3pt off 3pt,color=ffqqqq] (14.99,-2)-- (5,-2);
\draw [dash pattern=on 3pt off 3pt,color=ffqqqq] (5,-2)-- (5,0);
\draw [color=zzttqq] (11.92,1.22)-- (11.87,-2.83);
\draw [color=zzttqq] (11.87,-2.83)-- (8.67,-6.14);
\draw [color=zzttqq] (8.67,-6.14)-- (3.04,-5.54);
\draw [color=zzttqq] (3.04,-5.54)-- (0.73,-2.21);
\draw [color=zzttqq] (0.73,-2.21)-- (0.56,-1.17);
\draw [color=zzttqq] (0.56,-1.17)-- (3.3,4.23);
\draw [color=zzttqq] (3.3,4.23)-- (9.77,4.05);
\draw [color=zzttqq] (9.77,4.05)-- (11.92,1.22);
\fill [color=qqqqff] (14,2) circle (1.5pt);
\draw[color=qqqqff] (14.42,2.56) node {$a$};
\fill [color=qqqqff] (14,-4) circle (1.5pt);
\draw[color=qqqqff] (14.42,-4.4) node {$b$};
\fill [color=qqqqff] (6,-4) circle (1.5pt);
\draw[color=qqqqff] (6.35,-4.49) node {$c$};
\fill [color=qqqqff] (6,2) circle (1.5pt);
\draw[color=qqqqff] (6.41,2.56) node {$d$};
\fill [color=ffqqqq] (6,-0.98) ++(-1.5pt,0 pt) -- ++(1.5pt,1.5pt)--++(1.5pt,-1.5pt)--++(-1.5pt,-1.5pt)--++(-1.5pt,1.5pt);
\draw[color=ffqqqq] (6.41,-0.44) node {$z'$};
\fill [color=ffqqqq] (14,-0.98) ++(-1.5pt,0 pt) -- ++(1.5pt,1.5pt)--++(1.5pt,-1.5pt)--++(-1.5pt,-1.5pt)--++(-1.5pt,1.5pt);
\draw[color=ffqqqq] (13.46,-0.47) node {$w'$};
\fill [color=ffqqqq] (8.03,-0.98) ++(-1.5pt,0 pt) -- ++(1.5pt,1.5pt)--++(1.5pt,-1.5pt)--++(-1.5pt,-1.5pt)--++(-1.5pt,1.5pt);
\draw[color=ffqqqq] (8.36,-0.44) node {$z$};
\draw[color=ffqqqq] (3.59,-0.83) node {$R_z$};
\draw[color=zzttqq] (1.46,3.16) node {$\partial Q$};
\end{tikzpicture}

%% file: complex-figure7v2.tex
\definecolor{zzttqq}{rgb}{0.6,0.2,0}
\definecolor{ffqqqq}{rgb}{1,0,0}
\definecolor{qqqqff}{rgb}{0,0,1}
\begin{tikzpicture}[line cap=round,line join=round,>=triangle 45,x=0.5cm,y=0.5cm]
\clip(-3.2,-3.24) rectangle (3.5,6.1);
\draw[color=zzttqq,fill=zzttqq,fill opacity=0.1] (0,0.58) -- (0.42,0.58) -- (0.42,1) -- (0,1) -- cycle; 
\draw [shift={(0,0)},color=zzttqq,fill=zzttqq,fill opacity=0.1] (0,0) -- (90:0.6) arc (90:161.57:0.6) -- cycle;
\draw [shift={(1,-3)},color=zzttqq,fill=zzttqq,fill opacity=0.1] (0,0) -- (63.43:0.6) arc (63.43:135:0.6) -- cycle;
\draw [shift={(2,6)},color=zzttqq,fill=zzttqq,fill opacity=0.1] (0,0) -- (-135:0.6) arc (-135:-78.69:0.6) -- cycle;
\draw [shift={(0,3)},color=zzttqq,fill=zzttqq,fill opacity=0.1] (0,0) -- (-146.31:0.6) arc (-146.31:-90:0.6) -- cycle;
\draw [line width=1.2pt,color=qqqqff] (-3,1)-- (3,1);
\draw [line width=1.2pt,color=qqqqff] (3,1)-- (2,6);
\draw [line width=1.2pt,color=qqqqff] (2,6)-- (-3,1);
\draw [line width=1.2pt,color=qqqqff] (3,1)-- (1,-3);
\draw [line width=1.2pt,color=qqqqff] (1,-3)-- (-3,1);
\draw [line width=1.2pt,color=qqqqff] (-3,1)-- (3,1);
\draw (0,3)-- (2,6);
\draw (0.85,4.6) -- (1.15,4.4);
\draw (0,3)-- (3,1);
\draw (1.6,2.15) -- (1.4,1.85);
\draw (0,3)-- (-3,1);
\draw (-1.4,1.85) -- (-1.6,2.15);
\draw (3,1)-- (0,0);
\draw (1.56,0.43) -- (1.5,0.6);
\draw (1.5,0.4) -- (1.44,0.57);
\draw (0,0)-- (-3,1);
\draw (-1.5,0.4) -- (-1.44,0.57);
\draw (-1.56,0.43) -- (-1.5,0.6);
\draw (0,0)-- (1,-3);
\draw (0.57,-1.44) -- (0.4,-1.5);
\draw (0.6,-1.5) -- (0.43,-1.56);
\draw [dash pattern=on 2pt off 2pt,color=ffqqqq] (0,3)-- (0,0);
\draw [shift={(0,0)},color=zzttqq] (90:0.6) arc (90:161.57:0.6);
\draw [shift={(0,0)},color=zzttqq] (90:0.5) arc (90:161.57:0.5);
\draw [shift={(1,-3)},color=zzttqq] (63.43:0.6) arc (63.43:135:0.6);
\draw [shift={(1,-3)},color=zzttqq] (63.43:0.5) arc (63.43:135:0.5);
\fill [color=qqqqff] (-3,1) circle (1.5pt);
\draw[color=qqqqff] (-2.78,0.22) node {$z_1$};
\fill [color=qqqqff] (3,1) circle (1.5pt);
\draw[color=qqqqff] (3.18,0.16) node {$z_3$};
\fill [color=qqqqff] (2,6) circle (1.5pt);
\draw[color=qqqqff] (-1.3,4.12) node {$B$};
\fill [color=qqqqff] (1,-3) circle (1.5pt);
\fill [color=ffqqqq] (0,3) ++(-1.5pt,0 pt) -- ++(1.5pt,1.5pt)--++(1.5pt,-1.5pt)--++(-1.5pt,-1.5pt)--++(-1.5pt,1.5pt);
\draw[color=ffqqqq] (0.6,3.02) node {$z_2$};
\draw[color=black] (1.82,2.48) node {$Q$};
\fill [color=ffqqqq] (0,0) ++(-1.5pt,0 pt) -- ++(1.5pt,1.5pt)--++(1.5pt,-1.5pt)--++(-1.5pt,-1.5pt)--++(-1.5pt,1.5pt);
\draw[color=ffqqqq] (0.64,-0.44) node {$z_4$};
\draw[color=zzttqq] (1.2,-1.72) node {$\beta$};
\draw[color=zzttqq] (1.8,4.76) node {$\alpha$};
\end{tikzpicture}

%% file: complex.bbl
\begin{thebibliography}{99}

\bibitem{Alexa-Wardetzky-11} M. Alexa, M. Wardetzky, Discrete Laplacians on general polygonal meshes,
 ACM Trans. Graph. 30:4 (2011), 102:1--102:10.


\bibitem{Bobenko-etal-05} A.~I.~Bobenko, C.~Mercat, and Y.~B.~Suris, Linear and nonlinear theories of discrete analytic functions. Integrable structure and isomonodromic Green's function, J. Reine Angew. Math. \textbf{583} (2005), 117--161.

\bibitem{Bobenko-etal-09} A.~I.~Bobenko, C.~Mercat, and M.~Schmies, Period matrices of polyhedral surfaces. In: Computational Approach to Riemann Surfaces, A.~I.~Bobenko, C.~Klein (Eds), Lect. Notes Math. 2013,  Springer, Berlin, 2011, 213--226;  \url{http://arxiv.org/abs/0909.1305}.

\bibitem{Bobenko-etal-10} A.~I.~Bobenko, U.~Pinkall, and  B.~A.~Springborn, Discrete conformal maps and ideal hyperbolic polyhedra, preprint (2010); \url{http://arxiv.org/pdf/1005.2698v1}.

\bibitem{Bobenko-Springborn-07} A.~I.~Bobenko and B.~A.~Springborn, A discrete Laplace--Beltrami operator for simplicial surfaces, \textit{Discrete Comput. Geom.} \textbf{38} (2007), 740--756.




\bibitem{Braess-07} D.~Braess, Finite elements. Theory, fast solvers, and applications in elasticity theory, transl. by L.~L.~Schumaker, Cambridge Univ. Press, 2007.

\bibitem{Brandts-etal-11} J.~Brandts, A.~Hannukainen,
S.~Korotov, M.~K\v ri\v zek, \textrm{On angle conditions in the finite element method}, SeMA J \textbf{56} (2011), 81--95.





\bibitem{Chelkak-Smirnov-08} D.~Chelkak and S.~Smirnov,
\textrm{Discrete complex analysis on isoradial graphs},
\textit{Adv. Math.} 228 (2011), 1590-1630, \url{http://arxiv.org/abs/0810.2188v2}.

\bibitem{Ciarlet-78} P.~G.~Ciarlet, The finite element method for elliptic problems, North-Holland, Amsterdam, 1978. 

\bibitem{Ciarlet-Raviart-73} P.~G.~Ciarlet and P.-A.~Raviart, Maximum principle and uniform convergence for the finite element method, \textit{Computer Methods Appl. Mech. Engin.} \textbf{2} (1973), 17--31.




\bibitem{Courant-friedrichs-Lewy-28}
R.~Courant, K.~Friedrichs, H.~Lewy, \"Uber die partiellen Differenzengleichungen der mathematischen Physik, Math. Ann., \textbf{100}, (1928), 32--74. English transl.: IBM Journal (1967), 215--234.
Russian transl.: Russ.~Math.~Surveys \textbf{8} (1941), 125--160. \url{http://www.stanford.edu/class/cme324/classics/courant-friedrichs-lewy.pdf}






\bibitem{Duffin-59} R.~J.~Duffin, \textrm{Distributed and lumped networks}, \textit{J. Math. Mech.} \textbf{8:5} (1959), 793--826.

\bibitem{Duffin-68} R.~J.~Duffin, \textrm{Potential Theory on a Rhombic Lattice}, \textit{J. Combin. Theory} \textbf{5} (1968), 258--272.




\bibitem{Dynnikov-Novikov-03} I.~A.~Dynnikov and S.~P.~ Novikov, Geometry of the triangle equation on two-manifolds, \textit{Moscow Math. J.}, \textbf{3} (2003), 419--438.


\bibitem{Ferrand-44} J. Ferrand, Fonctions pr\'eharmoniques et fonctions pr\'eholomorphes, Bull. Sci. Math. \textbf{68} (1944), 152--180.




\bibitem{Foster} R.~M.~Foster, \textrm{Academic and Theoretical Aspects of Circuit Theory}, Proc. IRE \textbf{50:5} (1962), 866--871.








\bibitem{Hilderbrandt-etal-06} K.~Hildebrandt, K.~Polthier, M.~Wardetzky, \textrm{On the convergence of metric and geometric properties of polyhedral surfaces}, Geom. Dedicata \textbf{123} (2006), 89--112.

\bibitem{Hilderbrandt-etal-11} K.~Hildebrandt, K.~Polthier, On approximation of the Laplace--Beltrami operator and the Willmore energy of surfaces, Computer Graphics Forum \textbf{30} (2011), 1513--1520. 

\bibitem{Isaacs-41} R.
~Isaacs, A finite difference function theory, Univ. Nac. Tucum\'an. Revista A. \textbf{2} (1941), 177--201.




\bibitem{Kenyon-02} R.~Kenyon, \textrm{The Laplacian and Dirac operators on critical planar graphs},
    Invent.~Math.~\textbf{150:2} (2002), 409--439.



\bibitem{Lawler-Schramm-Werner-04} G.~F.~Lawler,
    O.~Schramm, W.~Werner, Conformal invariance of planar loop-erased random walks and uniform spanning trees, \textit{Ann. Probab.} \textbf{32:1B} (2004), 939--995.



\bibitem{Lovasz-04} L.~Lovasz, Discrete analytic functions: an exposition, In: Surv. Differ. Geom., IX, 241--273, Int. Press, Somerville, MA, 2004.

\bibitem{Lusternik-26}
L.~Lusternik,
 \"Uber einige Anwendungen der direkten Methoden in Variationsrechnung,
Sb Math+ 33:2 (1926), 173--202.

\bibitem{Mercat-01} C.~Mercat, Discrete Riemann surfaces and the Ising model, \textit{Comm. Math. Phys.} \textbf{218:1} (2001), 177--216.

\bibitem{Mercat-08} C.~Mercat, Discrete complex structure on surfel surfaces. In: Discrete Geometry for Computer Imagery, Coeurjolly et al. (Eds.), Lect. Notes Computer Sc. \textbf{4992} (2008), 153--164.




\bibitem{Prasolov-Skopenkov-11} M.~Prasolov and M.~Skopenkov, \textrm{Tilings by rectangles and alternating current}, \textit{J. Combin. Theory A} \textbf{118:3} (2011), 920--937, 
    \url{http://arxiv.org/abs/1002.1356}.






\bibitem{Saloff-Coste-97} L. Saloff-Coste, Some inequalities for superharmonic functions on graphs, Potential Analysis 6 (1997) 163--181.

\bibitem{Dorichenko-Prasolov-Skopenkov} M.~Skopenkov, M.~Prasolov, S.~Dorichenko, \textrm{Dissections of a metal rectangle}, Kvant \textbf{3} (2011), 10--16 (in Russian), \url{http://arxiv.org/abs/1011.3180}.

\bibitem{Skopenkov-Smykalov-Ustinov} M.~Skopenkov, V.~ Smykalov, A.~Ustinov, Random walks and electric networks, Matematicheskoe Prosveschenie 3rd ser. 16 (2012), 25--47, \url{http://www.mccme.ru/free-books/matprosh.html}.

\bibitem{Smirnov-10} S.~Smirnov, \textrm{Discrete Complex Analysis and Probability}, Proc. Intern. Cong. Math.
Hyderabad, India, 2010, \url{http://arxiv.org/abs/1009.6077}.







\bibitem{Wardetzky-etal-07} M.~Wardetzky, S.~Mathur,    F.~K\"aberer, E.~Grinspun, Discrete Laplace operators: no free lunch, Eurographics Symp. Geom. Processing, A.~Belyaev, M.~Garland (eds.), 2007.


\end{thebibliography}
